\documentclass[leqno]{amsart}
\usepackage{amsmath} 
\usepackage{amsthm}
\usepackage{amssymb}
\usepackage[foot]{amsaddr}
\usepackage{graphicx} % Required for inserting images
\usepackage[margin=1in]{geometry}
\usepackage[colorlinks = true, linkcolor = blue, citecolor = gray, urlcolor = blue]{hyperref}
\usepackage{xcolor}
\usepackage{breqn}
\usepackage{soul}

\renewcommand{\P}{\mathbb{P}}
\newcommand{\F}{\mathbb{F}}
\newcommand{\Z}{\mathbb{Z}}
\newcommand{\Q}{\mathbb{Q}}
\newcommand{\R}{\mathbb{R}}
\newcommand{\rhoELS}{\rho^{\mathrm{ELS}}}

\DeclareMathOperator{\Gal}{Gal}
\DeclareMathOperator{\Span}{span}
\DeclareMathOperator{\Ht}{ht}

\DeclareMathOperator{\Gr}{Gr}

\newtheorem{theorem}{Theorem}[section]
\newtheorem{conjecture}[theorem]{Conjecture}
\newtheorem{lemma}[theorem]{Lemma}
\newtheorem{proposition}[theorem]{Proposition}
\theoremstyle{definition}
\newtheorem{remark}[theorem]{Remark}
\newtheorem{definition}[theorem]{Definition}
\newtheorem{example}[theorem]{Example}

\numberwithin{equation}{section}
\numberwithin{table}{section}

\newcommand{\xival}[3][]{\xi^{(#1)}_{#2,#3}}
\newcommand{\fI}{f_{\mathrm{I}}}
\newcommand{\fII}{f_{\mathrm{II}}}
\newcommand{\fIII}{f_{\mathrm{III}}}

\makeatletter
\def\blfootnote{\gdef\@thefnmark{}\@footnotetext}
\makeatother

\begin{document}

\title{How often does a cubic hypersurface have a rational point?}

\author[Lea Beneish]{Lea Beneish$^1$}
\address{$^1$Department of Mathematics, University of North Texas, Denton, TX, USA}
\email{lea.beneish@unt.edu}

\author[Christopher Keyes]{Christopher Keyes$^2$}
\address{$^2$Department of Mathematics, King's College London, London, UK\newline \indent \hspace{1ex} and Heilbronn Institute for Mathematical Research, Bristol, UK}
\email{christopher.keyes@kcl.ac.uk}

\begin{abstract}
    A cubic hypersurface in $\mathbb{P}^n$ defined over $\mathbb{Q}$ is given by the vanishing locus of a cubic form $f$ in $n+1$ variables. It is conjectured that when $n \geq 4$, such cubic hypersurfaces satisfy the Hasse principle. This is now known to hold on average due to recent work of Browning, Le Boudec, and Sawin. Using this result, we determine the proportion of cubic hypersurfaces in $\P^n$, ordered by the height of $f$, with a rational point for $n \geq 4$ explicitly as a product over primes $p$ of rational functions in $p$. In particular, this proportion is equal to 1 for cubic hypersurfaces in $\mathbb{P}^n$ for $n \geq 9$; for $100\%$ of cubic hypersurfaces, this recovers a celebrated result of Heath-Brown that non-singular cubic forms in at least 10 variables have rational zeros. In the $n=3$ case, we give a precise conjecture for the proportion of cubic surfaces in $\P^3$ with a rational point.
\end{abstract}

\maketitle

\section{Introduction}

\blfootnote{\textit{License}: for the purposes of open access, the authors have applied a Creative Commons Attribution (CC BY) license to any Author Accepted Manuscript version arising from this submission.}

A cubic hypersurface in $\P^n$ defined over $\Q$ is given by the vanishing locus of an integral cubic form $f$ in $n+1$ variables, 
\begin{equation}\label{eq:cubic_form}
    X_f \colon f(x_0, \ldots, x_n) = \sum\limits_{0\leq i \leq j \leq k \leq n} a_{ijk}x_ix_jx_k = 0.
\end{equation}
We are interested in the proportion of $X_f$ that possess a rational point. In studying the rational points of such hypersurfaces, it is often useful to study their local points. If a variety $X/\Q$ possesses $\Q_v$-points for all places $v$ of $\Q$, we say $X$ is \textit{everywhere locally soluble}. This property is a necessary but not sufficient condition for $X$ to possess rational points. If $X$ is everywhere locally soluble, but still fails to have a rational point, we say there is an obstruction to the \textit{Hasse principle} (i.e., the Hasse principle does not hold).

More generally, a degree $d$ hypersurface in $\P^n$ is \textit{Fano} if $d \leq n$. Over any number field, it is known that smooth Fano hypersurfaces of dimension at least $3$ do not have a Brauer--Manin obstruction to the Hasse principle \cite[Appx.\ A]{poonenvoloch}. It has further been conjectured by Colliot-Th\'{e}l\`ene \cite{Colliot} that this is the only possible obstruction to the Hasse principle, hence the Hasse principle should hold. Browning, Le Boudec, and Sawin \cite{BLeBS} prove that this is \textit{almost always true} over the rationals in the following precise sense: in the limit as $A$ tends to infinity, the proportion of Fano hypersurfaces defined over $\mathbb{Q}$ of fixed degree $d$ in $\P^n$ (except in the case $n=d=3$) with height at most $A$ is equal to the proportion of such hypersurfaces that are everywhere locally soluble. 

Explicitly, counting hypersurfaces by the Euclidean height of the vector $\boldsymbol{a} \in \Z^{\binom{n+d}{d}}$ of coefficients of the defining degree $d$ form $f \in \Z[x_0, \ldots, x_n]$,
    \[\Ht(f) := || \boldsymbol{a} ||_2 =  \left( \sum_{0 \leq i_1 \leq \ldots \leq i_d \leq n} a_{i_1\ldots i_d}^2\right)^{1/2},\]
we define the natural densities
\begin{align}
	\label{eq:def_rho} \rho_{d,n} &= \lim_{A \to \infty} \frac{\# \{f : \Ht(f) \leq A,\ X_f(\mathbb{Q}) \neq \emptyset \}}{\#\{f : \Ht(f) \leq A\}}\\
	\label{eq:def_rhoELS} \rhoELS_{d,n} &= \lim_{A \to \infty} \frac{\# \{f : \Ht(f) \leq A,\ X_f(\mathbb{Q}_v) \neq \emptyset \text{ for all } v \}}{\#\{f : \Ht(f) \leq A\}}
\end{align}
where $X_f$ denotes the degree $d$ hypersurface in $\mathbb{P}^n$ cut out by $f=0$ and $v$ runs through all places of $\mathbb{Q}$. For fixed $(d,n) \neq (3,3)$ satisfying $2 \leq d \leq n$, Browning, Le Boudec, and Sawin show that  $\rho_{d,n} = \rho_{d,n}^{\mathrm{ELS}}$ \cite[Theorem 1.1]{BLeBS}. They also deduce that for $(d,n) \neq (2,2)$ satisfying $2 \leq d \leq n$, we have $\rho_{d,n} > 0$ \cite[Corollary 1.2]{BLeBS}; that is, for a positive proportion of degree $d$ Fano hypersurfaces $X_f \subset \P^n$ --- now excluding plane conics ($d=n=2$) but \textit{including} cubic surfaces ($d=n=3$) --- we have $X_f(\Q) \neq \emptyset$. 

In this paper, for each $n \geq 4$, we determine the proportion $\rho_{3,n}$ explicitly as a product of local factors and compute it numerically to high precision. A striking feature is that the local factors are uniform in $p$, given by rational functions.
\begin{theorem}\label{thm:main_thm}
 Let $n \geq 4$. There exist polynomials $g_n(t), h_n(t) \in \Z[t]$ such that the proportion of cubic hypersurfaces in $\P^n$ ordered by Euclidean height that possess a rational point is given by 
        \[\rho_{3,n} = \begin{cases}
            \displaystyle \prod_{p \text{ prime}} \left(1 - \frac{g_n(p)}{h_n(p)}\right) & 4 \leq n \leq 8 \\
            1 & n  \geq 9.
        \end{cases}\]
\noindent  The polynomials $g_n, h_n$ are given explicitly in \eqref{eq:g4} -- \eqref{eq:h8}.
\end{theorem}

In the case of cubic surfaces ($d=n=3$), $\rho_{3,3} = \rhoELS_{3,3}$ has been conjectured but is not known to hold \cite[Conjecture 2.2(ii)]{poonenvoloch}. We can, however, give a similarly explicit description of $\rhoELS_{3,3}$, leading to a conjectural formula for $\rho_{3,3}$.

\begin{conjecture}[Cubic surfaces]
    Let $n=3$. The proportion of cubic surfaces in $\P^3$ ordered by Euclidean height that possess a rational point is given by\footnote{Note that the conjectured formula for $\frac{g_3(p)}{h_3(p)}$ is not presented in lowest form, to save space.}
    \[\rho_{3,3} = \prod_p \textstyle \left(1 - \frac{\left(3  p^{26} + p^{24} + p^{23} + 4  p^{22} - 3  p^{21} + 3  p^{20} + 2  p^{19} + 2  p^{18} - p^{17} + p^{14} + p^{13} - 2  p^{12} + 3  p^{11} + 3  p^{7}\right) \left(p^{2} + 1\right) \left(p + 1\right)^{2} \left(p - 1\right)^{4}}{9  \left(p^{13} - 1\right) \left(p^{7} + 1\right) \left(p^{7} - 1\right) \left(p^{6} + 1\right) \left(p^{5} - 1\right) \left(p^{3} + 1\right) \left(p^{3} - 1\right)}\right).\]
    Numerically, $\rho_{3,3} \approx 0.999927$.
\end{conjecture}

\begin{remark}[asymptotics and numerics]
    We record below in Table \ref{tab:asymp_num} how quickly each local factor approaches 1, deduced from the explicit descriptions given in Section \ref{sec:explicit_functions}. We also record the approximate numerical values of $\rho_{3,n}$ (including the conjectural value of $\rho_{3,3}$). The details of these calculations and their precision are discussed in Section \ref{sec:numerics}.
    \begin{table}[h]
        \centering
        \caption{Asymptotics and numerics for $\rho_{3,n}$}
        \label{tab:asymp_num}        \renewcommand{\arraystretch}{1.25}
        \begin{tabular}{r c l}
            \hline $n$ & $\frac{g_n(t)}{h_n(t)} \sim$ & $\rho_{3,n} \approx$ \\ \hline
            3 & $1/3t^{10}$ & $0.999927$ (\textit{conj}.) \\
            4 & $1/9t^{22}$ & $1-5.022 \cdot 10^{-9}$ \\
            5 & $1/9t^{43}$ & $1-1.343 \cdot 10^{-15}$ \\
            6 & $1/9t^{78}$ & $1-3.502 \cdot 10^{-26}$ \\
            7 & $1/27t^{129}$ & $1-5.152 \cdot 10^{-42}$ \\
            8 & $1/27t^{201}$& $1-6.222 \cdot 10^{-64}$ \\ \hline
        \end{tabular}
    \end{table}
\end{remark}

For $n \geq 9$, Theorem \ref{thm:main_thm} recovers a consequence of Heath-Brown's celebrated result that non-singular cubic forms in $n+1 \geq 10$ variables have rational zeros \cite{HeathBrown}. By combining \cite{BLeBS} with a complete description of the probability that $X_f(\Q_p) \neq \emptyset$, we are able to further determine $\rho_{3,n}$ for $n \geq 4$ (see Theorem \ref{thm:rho_p_all_n}). The existence of nontrivial $p$-adic zeros for cubic forms in 10 or more variables was established by Dem'yanov for $p \neq 3$ \cite{Demyanov} and later by Lewis for all primes $p$ \cite{Lewis1952}. When $n=8$, work of Hooley shows that the Hasse principle holds for nonsingular cubic forms \cite{Hooley1988}. Heath-Brown's result is sharp in that there exist cubic forms in $9$ variables which fail to be everywhere locally soluble; this is reflected in Theorem \ref{thm:main_thm} as we compute $\rho_{3,8} < 1$, and extended to all $n \geq 4$ (conjecturally $n \geq 3$).

In light of \cite{BLeBS}, our strategy is to determine $\rhoELS_{3,n}$, the proportion of everywhere locally soluble cubic hypersurfaces in $\mathbb{P}^n$. We are able to do this by applying work of Poonen and Voloch \cite{poonenvoloch} and Bright, Browning, and Loughran \cite{BBL} to describe $\rho_{d,n}^{\mathrm{ELS}}$ as a product of local factors. These local factors are essentially the probabilities that $X_f(\Q_v) \neq \emptyset$ for a randomly chosen cubic form $f$ over $\Q_v$, which we determine explicitly and uniformly for finite places $v$. There have been other recent works studying everywhere local solubility in various families of hypersurfaces (see, for example, Browning, Fisher--Ho--Park, Hirakawa--Kanamura \cite{Browning,fisherhopark, HirakawaKanamura}). 

Most notably, Theorem \ref{thm:main_thm} may be regarded as a cubic analogue of work of Bhargava, Cremona, Fisher, Jones, and Keating on the density of integral quadratic forms in $n+1$ variables with a nontrivial integral zero; in particular, they give an explicit description of $\rho_{2,n}$ as a product of local probabilities which are rational functions in $p$ \cite{BCFJK}.

Our methods also recover $\rhoELS_{3,2}$, the density of everywhere local soluble plane cubic curves, first computed by Bhargava, Cremona, and Fisher as a product of explicit local densities \cite{BCF_plane_cubic}. Less is known about the density of plane cubics with a global point, $\rho_{3,2}$. Bhargava showed both that a positive proportion of plane cubics fail the Hasse principle and that a positive proportion have a rational point \cite[Theorems 1, 2]{bhargava2014positive}, i.e.\ $0 < \rho_{3,2} < \rhoELS_{3,2}$. Combining his methods with conjectures on the distribution of ranks of elliptic curves over $\Q$, Bhargava further conjectured $\rho_{3,2} = \frac13\rhoELS_{3,2}$ \cite[Conjecture 6]{bhargava2014positive}.

This paper is organized as follows. In Section \ref{sec:ELS} we give an overview of the strategy of the paper. In Section \ref{sec:factorization} we analyze the possible ways in which a cubic form can factor over a finite field and determine their probabilities of occurrence in cases of interest. In Sections \ref{sec:lifting_1} and \ref{sec:lifting_not_1}, we compute the probability that a $p$-adic cubic hypersurface has a $\Q_p$-point given certain conditions on its reduction $\overline{X_f}$; the former handles the case where the reduction is not a configuration of conjugate hyperplanes over $\F_p$, while the latter handles precisely these cases. In Section \ref{sec:explicit_functions} we prove Theorem \ref{thm:main_thm} and give the explicit formulas for the local factors, and in Section \ref{sec:numerics} we describe how we obtain precise numerical values.

\subsection*{Acknowledgments} The authors are grateful to Jackson Morrow for bringing \cite{BLeBS} to their attention and would like to thank Tim Browning, Tom Fisher, Rachel Newton, and Bjorn Poonen for helpful comments on an earlier draft. CK was supported by the Additional Funding Programme for Mathematical Sciences, delivered by EPSRC (EP/V521917/1) and the Heilbronn Institute for Mathematical Research. 

\section{Everywhere local solubility}
\label{sec:ELS}

Let $X/\Q$ be a variety and $v$ denote a place of $\Q$.

\begin{definition}
    $X$ is \textbf{locally soluble at} $\boldsymbol{v}$ if $X(\Q_v) \neq \emptyset$ and \textbf{everywhere locally soluble} if $X(\Q_v) \neq \emptyset$ for all places $v$.
\end{definition}

It follows from the inclusions $X(\Q) \hookrightarrow X(\Q_v)$ that everywhere local solubility is necessary for $X(\Q) \neq \emptyset$. Thus in studying how often a degree $d$ hypersurface $X_f$ has a rational point, it will be useful to keep track of how often $X_f$ is everywhere locally soluble. We denote by $\rhoELS_{d,n}$ the natural density of integral degree $d$ forms $f$ (with respect to the Euclidean height $\Ht(f) = ||\boldsymbol{a}||_2$) for which the hypersurface $X_f \subset \P^n$ is everywhere locally soluble, given in \eqref{eq:def_rhoELS}.

The limit definition of $\rhoELS_{d,n}$ is unwieldy to compute with. However, it follows from work of Bright, Browning, and Loughran \cite[Theorem 1.4]{BBL} that when $n \geq 2$ and $(d,n) \neq (2,2)$,
\begin{equation}\label{eq:product_formula}
	\rhoELS_{d,n} = \rho_{d,n}(\infty) \prod_p \rho_{d,n}(p),
\end{equation}
where $\rho_{d,n}(v)$ is the density of $v$-adic degree $d$ hypersurfaces with a $v$-adic point. This is made precise for the finite places by letting $\mu_p$ to be the normalized Haar measure on $\Z_p^{\binom{n+d}{d}}$, the space of integral $p$-adic degree $d$ forms, and taking
\[\rho_{d,n}(p) = \mu_p \left( \left\{ f \in \Z_p[x_0, \ldots, x_n] : X_f(\Q_p) \neq \emptyset \right\} \right),\] 
where $f$ is identified with the tuple corresponding to its coefficients in $\Z_p^{\binom{n+d}{d}}$. 

In their original paper on random hypersurfaces, Poonen and Voloch proved \eqref{eq:product_formula} with a different choice of height function when defining $\rhoELS_{d,n}$ and $\rho_{d,n}(\infty)$ \cite[Theorem 3.6]{poonenvoloch}; see Remark \ref{rem:choices}. Similar product formulae for the densities of everywhere local solubility hold in other families of varieties. Poonen and Stoll showed that a version of \eqref{eq:product_formula} holds for the density of everywhere locally soluble hyperelliptic curves \cite{poonenstoll_short}, in work which predates \cite{poonenvoloch}. Bright, Browning, and Loughran generalized their approach to show the analogue of \eqref{eq:product_formula} holds for families coming from fibers of maps to affine or projective space, subject to certain geometric conditions \cite{BBL}. This was employed by the authors to prove an analogue of \eqref{eq:product_formula} for families of superelliptic curves \cite{BeneishKeyes_density}. In another direction, Fisher, Ho, and Park show that the analogue of \eqref{eq:product_formula} holds for families of hypersurfaces in products of projective space \cite[Theorem 1.1]{fisherhopark}.

\begin{remark}[choices when defining densities]
\label{rem:choices}    
    Let $\Psi \subset \R^{\binom{n+d}{d}}$ denote a bounded subset of positive measure with boundary measure zero, and define
    \[\rhoELS_{d,n,\Psi} = \lim_{A \to \infty} \frac{\# \{ f \in A\Psi \cap \Z^{\binom{n+d}{d}} : X_f \text{ ELS}\}}{\# \{ f \in A\Psi \cap \Z^{\binom{n+d}{d}}\}}.\]
    Let $\Psi' \subseteq \Psi$ denote the subset of $f$ for which $X_f(\R) \neq \emptyset$ and $\mu_\infty$ be the Lebesgue measure on $\R^{\binom{n+d}{d}}$. It follows from \cite[Proposition 3.2]{BBL} that
    \[\rhoELS_{d,n,\Psi} = \frac{\mu_\infty(\Psi')}{\mu_\infty(\Psi)} \prod_{p} \rho_{d,n}(p).\]
    We take $\rho_{d,n,\Psi}(\infty) = \frac{\mu_\infty(\Psi')}{\mu_\infty(\Psi)}$ to get \eqref{eq:product_formula} and note that the product over finite places does not depend on $\Psi$. 
    
    Different choices of $\Ht(f)$ correspond to different choices of $\Psi$. In their original paper, Poonen and Voloch used $\Ht(f) = ||\boldsymbol{a}||_\infty$, i.e.\ the maximum absolute value of the coefficients of $f$; this corresponds to taking $\Psi = [-1,1]^{\binom{n+d}{d}}$. The Euclidean height $\Ht(f) = ||\boldsymbol{a}||_2$, used here and in \cite{BLeBS}, corresponds to taking $\Psi$ to be a sphere of radius 1 in $\R^{\binom{n+d}{d}}$. In either case, $\rhoELS_{d,n,\Psi}/\rhoELS_{d,n,\Psi}(\infty)$ coincide. We also point out that in \cite{BLeBS}, Browning, Le Boudec, and Sawin count only \textit{primitive} $f$, i.e.\ those whose coefficients have no common divisor. By a standard M\"obius inversion argument, the densities $\rho_{d,n}$ and $\rhoELS_{d,n}$ are unchanged whether we choose to count all integral $f$ or just primitive $f$.
\end{remark}

Returning to families of cubic hypersurfaces, the local factors $\rho_{3,n}(v)$ can be computed explicitly. Since real cubic forms possess real zeros, we have $\rho_{3,n}(\infty) = 1$. For the finite places $v=p$, we prove that these local probabilities are given in terms of rational functions, \textit{uniformly} in $p$.

\begin{theorem}\label{thm:rho_p_all_n}
    Let $n \geq 1$. There exist $g_n(t), h_n(t) \in \Z[t]$ such that for all primes $p$ we have
    \[\rho_{3,n}(p) = \begin{cases}
    	1 - \frac{g_n(p)}{h_n(p)} & 1 \leq n \leq 8\\
    	1 & n \geq 9,
    	\end{cases} \]
    with $g_n, h_n$ given explicitly in \eqref{eq:g1} -- \eqref{eq:h8}.
\end{theorem}

In the following sections we set out to compute $\rho_{3,n}(p)$ for $n \geq 1$, thereby proving Theorem \ref{thm:rho_p_all_n}. Theorem \ref{thm:main_thm} then follows by applying \eqref{eq:product_formula} and \cite{BLeBS}. The strategy is to reduce modulo $p$ and wherever possible find points on $\overline{X_f}(\F_p)$ which may be lifted to $p$-adic points via Hensel's lemma. If $\overline{X_f}(\F_p)$ consists only of singular points, we apply a transformation and repeat this process of reducing and lifting. Along the way, we need to keep track of various \textit{factorization} and \textit{lifting} probabilities.

We begin in Section \ref{sec:factorization} by measuring how often a cubic form $f$ over $\F_p$ factors so that the associated hypersurface over $\F_p$ decomposes into a configuration of hyperplanes conjugate over $\F_{p^3}$; these are the aforementioned factorization probabilities. A key insight, made in Section \ref{sec:lifting_1}, is that when the reduction $\overline{X_f}$ is not one of these distinguished configurations, $X_f(\F_p)$ contains a point which lifts via Hensel's lemma to a point in $X_f(\Q_p)$. In these cases, the lifting probability is simply 1. When $\overline{X_f}$ is a configuration of conjugate hyperplanes, it possesses only singular $\F_p$-points, and the lifting probabilities present a challenge. In Section \ref{sec:lifting_not_1}, we write $\rho_{3,n}(p)$ as a linear combination of the appropriate factorization and lifting probabilities and investigate relations between them. This culminates in a large system of relations which can be solved symbolically by the computer algebra system \texttt{Sage} \cite{sagemath}, completing the proof of Theorem \ref{thm:rho_p_all_n} by obtaining the expressions \eqref{eq:g1} -- \eqref{eq:h8} for $n \leq 8$ and verifying $\rho_{3,n}(p) = 1$ when $n \geq 9$.

The case of binary cubic forms ($n=1$) was already known; see e.g.\ \cite[\S 1.2.3]{BCFG}. In the case of plane cubic curves ($n=2$), the conclusion of Theorem \ref{thm:rho_p_all_n} was shown by Bhargava, Cremona, and Fisher \cite{BCF_plane_cubic}. Indeed, for $n=1,2$, the $g_n(t)$ and $h_n(t)$ that we produce agree with these known results. While our overall strategy is similar to theirs in spirit, in dimensions $n > 2$ certain factorization cases require extra attention, and our approach requires three ``phases" rather than the two needed for the planar situation. Furthermore, we choose to block the variables $x_0, \ldots, x_n$ in a way that allows us to work essentially uniformly in $n$. For that reason, in what follows we allow $n=1,2$ and make note of similarities and differences for the reader familiar with \cite{BCF_plane_cubic}.

A feature of interest in this approach is that it is entirely uniform in $p$. The same cannot be said for local solubility problems in other families, including genus one hyperelliptic curves \cite{BCF_genus_one} and genus four trigonal superelliptic curves \cite{BeneishKeyes_density}, where the local densities are given by rational functions only for sufficiently large $p$. In the latter case, those rational functions also depend on the residue class of $p$ modulo 3. 

\begin{remark}[finite extensions]
    Theorem \ref{thm:rho_p_all_n} can be extended to finite extensions $K/\Q$ as follows. For a completion $K_v$ at a finite place $v$, let $\mathcal{O}_v$ denote the valuation ring and $\F_v$ the residue field with $q = \#\F_v$. At each finite place $v$, the $v$-adic Haar measure of the set of cubic forms $f \in \mathcal{O}_v[x_0, \ldots, x_n]$ with $X_f(\mathcal{O}_v) \neq \emptyset$ is denoted $\rho_{3,n}(v)$ and given explicitly by the same formula
    \[\rho_{3,n}(v) = 1 - \frac{g_n(q)}{h_n(q)}.\]
    Aside from the factorization probabilities presented in Section \ref{sec:factorization} for arbitrary finite fields $\F_q$, we elect to restrict to the setting of $K=\Q$, since this is where we have $\rho_{3,n} = \rhoELS_{3,n} = \prod_p \rho_{3,n}(p)$ for $n \geq 4$ due to \cite{BLeBS}. However, the proof of Proposition \ref{prop:lifting_prob_1} and the strategy in Section \ref{sec:lifting_not_1} generalize readily to $K_v$.
\end{remark}

\section{Factorization probabilities}
\label{sec:factorization}

Fix a prime power $q$ and let $\F_q$ denote the finite field with $q$ elements. Let $f \in \F_q[x_0, \ldots, x_n]$ be a nonzero cubic form. In Table \ref{tab:factorizations} below, we record the possible factorizations of $f$ over the algebraic closure $\overline{\F_q}$ along with their geometric descriptions. Here each (decorated) $\ell$ denotes a distinct linear form, $g$ denotes a quadratic form, and $\sigma$ denotes the Frobenius action generating the Galois group of an extension $\F_{q^r}/\F_q$.
\begin{table}[h]
\centering
\caption{Possible geometric factorizations of a cubic hypersurface over $\F_{q}$}
\label{tab:factorizations}
\renewcommand{\arraystretch}{1.25}
\begin{tabular}{l l l}
	\hline Symbol & Factorization & Description \\ \hline 
	$(1^3)$ & $f = \ell^3$ & triple hyperplane over $\F_q$\\
	$(1^21)$ & $f = \ell^2 \ell'$ & double hyperplane and hyperplane both over $\F_q$\\
	$(111)$ & $f = \ell_1 \ell_2 \ell_3$ & three distinct hyperplanes over $\F_q$ \\
	$(111)_2$ & $f = \ell (\sigma \ell) \ell'$ & two conjugate hyperplanes over $\F_{q^2}$ and hyperplane over $\F_q$ \\
	$(111)_3 $ & $f = \ell (\sigma \ell)(\sigma^2 \ell)$ & three conjugate hyperplanes over $\F_{q^3}$ \\
	$(21)$ & $f = g\ell$ & quadric hypersurface and hyperplane both over $\F_q$\\
	$(3)$ & $f$ & geometrically irreducible cubic hypersurface over $\F_q$\\ \hline
\end{tabular}
\end{table}

We are especially interested in the case where $f$ factors over $\overline{\F_q}$ as the product of conjugates of a linear form over $\F_{q^3}$, i.e.\ those of the form $(1^3)$ or $(111)_3$ in Table \ref{tab:factorizations}. When $n=2$, $(1^3)$ corresponds to a \textit{triple line}, and $(111)_3$ can be either a \textit{star}, or a  \textit{triangle} as described in \cite{BCF_plane_cubic}. We can make these precise for general $n$ as follows.

\begin{definition}\label{def:types}
    Let $f \in \F_q[x_0, \ldots, x_n]$ be a nonzero cubic form. Suppose there exists a linear form $\ell = b_0x_0 + \ldots + b_nx_n \in \F_{q^3}[x_0, \ldots, x_n]$ such that
    \[f = \prod_{\sigma \in \Gal(\F_{q^3}/\F_q)} \sigma \ell\]
    and whose coefficients span a dimension $i$ $\F_q$-subspace of $\F_{q^3}$,
    \[i = \dim_{\F_q} \Span\{ b_0,\ldots, b_n\}.\]
    Then we say $f$ has \textbf{factorization type} $\boldsymbol{i}$.
\end{definition}

The factorization type of $f$ is invariant under linear change of coordinates over $\F_q$ and scaling $f$ by an element of $\F_q^\times$. Type 1 is equivalent to $f = \ell^3$, or $X_f$ a hyperplane with multiplicity 3. Types 2 and 3 correspond to the $(111)_3$ factorization of Table \ref{tab:factorizations}, or $X_f$ composed of three conjugate hyperplanes; the pairwise intersections coincide for type 2, while they do not for type 3 (the generic case).

\begin{remark}\label{rem:isolate_vars}
	Suppose $f\in \F_q[x_0, \ldots, x_n]$ is a nonzero cubic form with factorization type $i \in \{1,2,3\}$. There is a bijection between $X_f(\F_q)$ and $\P^{n-i}(\F_q)$. This can be seen by making a linear change of variables resulting in $i = \dim_{\F_q} \Span\{ b_0,\ldots, b_{i-1}\}$ and $b_j = 0$ for $j \geq i$ and seeing that 
	\[X_f(\F_q) = \{[0 : \ldots : 0 : x_i : \ldots : x_n] : [x_i : \ldots : x_n] \in \P^{n-i}(\F_q)\}.\]
\end{remark}

Let $N_n$ denote the number of nonzero cubic forms $f \in \F_q[x_0, \ldots, x_n]$ and $N_{n,i}$ the number of such forms having factorization type $i$ for $i=1,2,3$. We take $\xi_{n,i} = N_{n,i}/N_n$ to be the probability that a nonzero cubic form has type $i$; this is equal to the probability (in the sense of Haar measure) that a primitive cubic form $f \in \Z_p[x_0, \ldots, x_n]$ has reduction $\overline{f}$ with type $i$. Let us further define
\[\xi_{n,0} = 1 - \xi_{n,1} - \xi_{n,2} - \xi_{n,3},\]
the probability of not being one of these three types.

\begin{lemma}\label{lem:xi_no_cond}
    Let $n \geq 0$. We have the following values for $\xi_{n,i}$. 
    \begin{align*}
        \xi_{n,0} &= 1 - \frac{q^{3n-3} + 2q^{n+3} + 2q^{n+2} + 2q^{n+1} - 2q^2 - 2q - 3}{3\left(q^2+q+1\right)\left(q^{\binom{n+3}{3}}-1\right)}\\ 
        \xi_{n,1} &= \frac{q^{n+1}-1}{q^{\binom{n+3}{3}} - 1}\\ 
        \xi_{n,2} &= \frac{{\left(q^{2  n + 1} - q^{n + 1} - q^{n} + 1\right)} q}{3\left(  q^{\binom{n+3}{3}} - 1\right)}\\
        \xi_{n,3} &= \frac{{\left(q^{3  n} - q^{2  n} - q^{2  n + 1} - q^{2  n - 1} + q^{n + 1} + q^{n - 1} + q^{n} - 1\right)} q^{3}}{3  {\left(q^{2} + q + 1\right)} {\left(q^{\binom{n+3}{3}} - 1\right)}}
    \end{align*}
\end{lemma}

\begin{proof}
    When $n=0$, we have $f = ax_0^3$ which has type 1. Thus $\xi_{0,1} = 1$ and $\xi_{0,0} = \xi_{0,2} = \xi_{0,3} = 0$, which are in agreement with the stated formulas. 

    From here on we assume $n \geq 1$. We have $N_n = q^{\binom{n+3}{3}} - 1$. The key idea is to count the linear forms into which $f$ factors appropriately, and account for multiplicity due to scaling and conjugation. To compute $N_{n,1}$, the number of cubic forms in $n+1$ variables that factor as a triple hyperplane, we count linear forms over $\F_q$. Since there are $\frac{q^{n+1}-1}{q-1}$ such hyperplanes, after accounting for possible scaling, we obtain 
    \[N_{n,1} = q^{n+1} - 1.\]

    The value of $N_{1,2}$ was computed in \cite[Lemma 5]{BCF_plane_cubic}. For general $n$, we compute $N_{n,2}$ by again counting linear forms $\ell = \sum_{k=0}^n b_k x_k$, but this time they are defined over $\F_{q^3}$ and their coefficients span a 2-dimensional $\F_q$-subspace. After scaling, we may assume that the first nonzero coefficient occurs at index $s$ and is $b_s = 1$, and the first coefficient $b_k \notin \F_q$ occurs at index $t$, for which there are $q^3 - q$ possibilities. Thus we have, after accounting for conjugating $\ell$ and scaling at the end, 
    \begin{align*}
        N_{n,2} &= \frac{q-1}{3}\sum_{0 \leq s < n} \sum_{s < t \leq n} q^{t-s-1}(q^3-q)q^{2(n-t)}\\
        % &= \frac{(q-1)(q^3-q)}{3} \sum_{0 \leq s < n} \sum_{s < t \leq n} q^{2n-s-t-1} \\
        &= \frac{(q-1)(q^3-q)}{3} \sum_{0 \leq s < n} q^{n-s-1}\left( \frac{q^{n-s} - 1}{q-1} \right)   \\      
        % &= \frac{q^3-q}{3} \sum_{0 \leq s < n} q^{2n - 2s-1} - q^{n-s-1}\\ 
        % &= \frac{q^3-q}{3} \left( \frac{q(q^{2n} - 1)}{q^2-1} - \frac{q^n-1}{q-1} \right)   \\
        % &= \frac{q^3-q}{3(q^2-1)} \left( q^{2n+1} - q^{n+1} - q^n + 1 \right) \\
        &= \frac{q \left( q^{2n+1} - q^{n+1} - q^n + 1 \right)}{3}.
    \end{align*}

    The value of $N_{2,3}$ also appeared implicitly in \cite[\S 2.2.3]{BCF_plane_cubic}. For general $n$, the calculation of $N_{n,3}$ is quite similar to that of $N_{n,2}$, with indices $s,t$ as above, and also a first coefficient $b_u$ that is not contained in $\Span_{\F_q}\{1,b_t\}$. Thus we compute
    \begin{align*}
        N_{n,3} &= \frac{q-1}{3} \sum_{0 \leq s < n-1}\sum_{s < t < n}\sum_{t < u \leq n} q^{t-s-1}(q^3-q)q^{2(u-t-1)}(q^3-q^2)q^{3(n-u)}\\
        % &= \frac{(q-1)(q^3-q)(q^3-q^2)}{3} \sum_{0 \leq s < n-1}\sum_{s < t < n}\sum_{t < u \leq n} q^{3n-s-t-u-3} \\
        &= \frac{(q-1)(q^3-q)(q^3-q^2)}{3} \sum_{0 \leq s < n-1}\sum_{s < t < n} q^{2n-s-t-3}\left( \frac{q^{n-t} - 1}{q-1} \right) \\
        % &= \frac{(q^3-q)(q^3-q^2)}{3} \sum_{0 \leq s < n-1}\sum_{s < t < n} q^{3n-s-2t-3} - q^{2n-s-t-3} \\
        &= \frac{(q^3-q)(q^3-q^2)}{3} \sum_{0 \leq s < n-1}\left( \frac{q^{n-s-1}(q^{2n-2s-2}-1)}{q^2-1} - \frac{q^{n-s-2}(q^{n-s-1}-1)}{q-1} \right)\\
        % &= \frac{q(q^3-q^2)}{3} \sum_{0 \leq s < n-1} \left( q^{3n-3s-3} - q^{2n-2s-2} - q^{2n-2s-3} + q^{n-s-2} \right)\\
        &= \frac{q(q^3-q^2)}{3}\left( \frac{q^3(q^{3n-3}-1)}{q^3-1} - \frac{q(q^{2n-2}-1)}{q-1} + \frac{q^{n-1}-1}{q-1} \right).
    \end{align*}
    Finally, we have $ N_{n,0} = N_n - N_{n,1} - N_{n,2} - N_{n,3}$ and the stated values of $\xi_{n,i}$ follow from $\xi_{n,i} = N_{n,i}/N_{n}$.
\end{proof}

\begin{remark}
	We can also interpret the counts $N_{n,i}$ for $i=1,2,3$ above in terms of the number of $\F_q$-points on certain Grassmannians. Recall $\Gr(r,k)$ is a variety whose $\F_q$-points parametrize $r$-dimensional subspaces of $\F_q^k$. For $n \geq i$, we have
	\[N_{n,i} = \# \Gr(n+1 - i,n+1)(\F_q) \cdot N_{i-1,i}.\]
	To see this, recall from Remark \ref{rem:isolate_vars} that a cubic form $f \in \F_q[x_0, \ldots, x_n]$ with factorization type $i$ has its geometric components intersect over a codimension $i$ linear subspace of $\P^n$ over $\F_q$. There are $\# \Gr(n+1 - i,n+1)(\F_q)$ choices for this subspace. Changing coordinates, we may assume this intersection is at $x_0 = \cdots = x_{i-1} = 0$. In order for $f = \prod \sigma(\ell)$ to vanish on this subspace, it must involve only the coordinates $x_0, \ldots, x_{i-1}$, and by definition there are $N_{i-1, i}$ such $f$ with factorization type $i$.
\end{remark}

\subsection{Additional conditions}

We continue with $f \in \F_q[x_0, \ldots, x_n]$ a nonzero cubic form and $X_f/\F_q$ the associated cubic hypersurface. Recalling our notation \eqref{eq:cubic_form}, we write $f = \sum_{0 \leq i \leq j \leq k \leq n} a_{ijk}x_ix_jx_k$.

\begin{definition}
    We say $f$ satisfies \textbf{condition (1)}, or the \textbf{point condition}, if $a_{000} \neq 0$, i.e.\ $X_f$ does not contain the $\F_q$-\textit{point} $[1:0:\ldots :0]$.
\end{definition}

\begin{definition}
    We say $f$ satisfies \textbf{condition (2)}, or the \textbf{line condition}, if $f(x_0, x_1, 0, \ldots, 0)$ is an irreducible binary cubic form over $\F_q$, i.e.\ the intersection of $X_f$ with the \textit{line} $x_2 = \ldots = x_n = 0$ contains no $\F_q$-points.
\end{definition}

In \cite{BCF_plane_cubic, BCFJK}, similar point and line conditions are used in recursive computations of the probability of $p$-adic solubility for plane cubic curves and quadratic forms. To deal with cubic hypersurfaces of higher dimension, we find we will need an additional condition.

\begin{definition}
   We say $f$ satisfies \textbf{condition (3)}, or the \textbf{plane condition}, if $f(x_0, x_1, x_2, 0, \ldots, 0)$ has factorization type 3 (the triangle configuration), i.e.\ the intersection of $X_f$ with the \textit{plane} $x_3 = \ldots = x_n = 0$ contains no $\F_q$-points.
\end{definition}

Let $\xival[j]{n}{i}$ denote the probability that a nonzero cubic form \textit{satisfying condition $(j)$} has factorization type $i$ for $1 \leq i,j \leq 3$. This is again equal to the probability (in the sense of Haar measure) that a primitive cubic form $f \in \Z_p[x_0, \ldots, x_n]$ with reduction $\overline{f}$ satisfying condition $j$ has factorization type $i$. As before, we further define
\[\xival[j]{n}{0} = 1 - \xival[j]{n}{1} - \xival[j]{n}{2} - \xival[j]{n}{3},\]
the probability of not being one of these three types. For $n=1,2$, $\xival[j]{n}{i}$ were computed in \cite[Proposition 7]{BCF_plane_cubic}.

\begin{lemma}\label{lem:xi_cond}
    For $n \geq 1$ and $j \leq n+1$ we have the following values for $\xival[j]{n}{i}$.
    \begin{center}
    \renewcommand{\arraystretch}{2.5}
    \begin{tabular}{| c | c  c  c |}
        \hline $i \backslash j$  & 1 & 2 & 3 \\ \hline
        0 & $\displaystyle 1 - \frac{q^{2n} + q^{n + 2} - q^{n} + 2}{3  q^{\binom{n+3}{3}-n-1}}$ 
          & $\displaystyle 1 - \frac{1}{q^{\binom{n+3}{3}-3n-1}}$
          & $\displaystyle 1 - \frac{1}{q^{\binom{n+3}{3}-3n-8}}$\\ 
        1 & $\displaystyle \frac{1}{q^{\binom{n+3}{3}-n-1}}$ 
          & 0 
          & 0 \\ 
        2 & $\displaystyle \frac{{\left(q + 1\right)} {\left(q^{n + 1} - 1\right)}}{3  q^{\binom{n+3}{3}-n-1}}$ 
          & $\displaystyle \frac{1}{q^{\binom{n+3}{3}-2n-2}}$
          & 0 \\  
        3 & $\displaystyle \frac{q^{2  n - 1} - q^{n - 1} - q^{n} + 1}{3  q^{\binom{n+3}{3} - n - 2}}$ 
          & $\displaystyle \frac{q^{n - 1} - 1}{q^{\binom{n+3}{3} - 2n - 2}}$
          & $\displaystyle \frac{1}{q^{\binom{n+3}{3}-3n-4}}$ \\ \hline 
    \end{tabular}
    \end{center}
\end{lemma}

\begin{proof}
    The proof is analogous to that of Lemma \ref{lem:xi_no_cond}. Let $N_{n}^{(j)}$ denote the number of cubic forms $f$ in $n+1$ variables subject to condition $(j)$ and $N_{n,i}^{(j)}$ denote the number of such forms with factorization type $i$.

    Now allowing the point, line, and plane conditions to be applied, we compute $N_n^{(j)}$ by recognizing that the condition imposes a constraint on the coefficients in the first $j$ variables, with the rest free to be chosen from $\F_q$. Explicitly, we find
    \begin{align*}
        N_n^{(1)} &= (q-1)q^{\binom{n+3}{3} - 1}, \\% = q^{\binom{n+3}{3}} - q^{\binom{n+3}{3} - 1}\\
        N_n^{(2)} &= N_{1,2}q^{\binom{n+3}{3} - 4} = \frac{1}{3}(q-1)^2(q+1)q^{\binom{n+3}{3} - 3},\\
        N_n^{(3)} &= N_{2,3}q^{\binom{n+3}{3} - 10} = \frac{1}{3}(q-1)^3(q+1)q^{\binom{n+3}{3} - 7}.
    \end{align*}
    
    To compute $N_{n,i}^{(j)}$, we essentially repeat the calculations in Lemma \ref{lem:xi_no_cond}. The condition $(j)$ has the effect of simplifying several steps, much in the same way that it is easier to count monic polynomials than all polynomials. When $0 < i < j$ we have $N_{n,i}^{(j)} = \xival[j]{n}{i}=0$, since condition $(j)$ imposes that the coefficients of the linear form into which $f$ factors (possibly over $\F_{q^3}$) span a $j$-dimensional $\F_q$-vector space, ruling out type $i$.
    Here are the values of $N_{n,i}^{(j)}$, for completeness.
    \begin{align*}
        N_{n,1}^{(1)} &= (q-1)q^{\binom{n+3}{3}-1}\\
        N_{n,2}^{(1)} &= \frac{q-1}{3} \sum_{0 < t \leq n} q^{t-1}(q^3-q)q^{2(n-t)}\\
            &= \frac{1}{3}q^n(q^2-1)(q^{n}-1)\\
        N_{n,2}^{(2)} &= \frac{1}{3}(q-1)(q^3-q)q^{2n-2}\\
        N_{n,3}^{(1)} &= \frac{q-1}{3}\sum_{0 < t \leq n-1} \sum_{t < u \leq n} q^{t-1}(q^3-q)q^{2(u-t-1)}(q^3-q^2)q^{3(n-u)}\\
            &= \frac{1}{3} q^{n+1}(q-1)(q^{2n-1}-q^n-q^{n-1} + 1)\\
        N_{n,3}^{(2)} &= \frac{q-1}{3} \sum_{1 < u \leq n} (q^3-q)q^{2(u-2)}(q^3-q^2)q^{3(n-u)}\\
            &= \frac{1}{3} q^{2n-1}(q-1)(q^2-1)(q^{n-1}-1)\\
        N_{n,3}^{(3)} &= \frac{1}{3}q^{3n-3}(q-1)^3(q+1)
    \end{align*}
    We then compute $N_{n,0}^{(j)}$ as before and take the ratios with $N_{n}^{(j)}$ to obtain the stated values.
\end{proof}

\begin{remark}
	When $n=0$, we have $\xi_{0,1}^{(1)} = 1$ and $\xi_{0,i}^{(1)} = 0$ for $i=0,2,3$. When $i=2$ this is not evident from the value presented in Lemma \ref{lem:xi_cond} due to extra cancellation. From the proof, however, we do see $N_{0,2}^{(1)} = 0$.
\end{remark}

\section{Lifting probabilities that are 1}
\label{sec:lifting_1}

Let $f \in \Z_p[x_0, \ldots, x_n]$ be a primitive cubic form. We can characterize those $f$ for which $X_f$ has no $p$-adic points as follows.
\begin{proposition}\label{prop:lifting_prob_1}
    Let $n \geq 1$ and suppose $X_f(\Q_p) = \emptyset$. Then $\overline{f}$ has factorization type 1, 2, or 3 described above in Definition \ref{def:types}
\end{proposition}

The proof will proceed by induction on $n$. We will need an intermediate result for the inductive step.

\begin{lemma}\label{lem:inductive step}
    Let $n \geq 3$ and suppose $f \in \F_q[x_0, \ldots, x_n]$ is a nonzero cubic form. Suppose that for all hyperplanes $H \subset \P^n$ defined over $\F_q$, we have either that $H$ is an irreducible component of $X_f$ or $X_f \cap H \subset \P^{n-1}$ is the union of conjugate hyperplanes over $\F_{q^3}$. Then $f$ has type 1, 2, or 3.
\end{lemma}

\begin{proof}
    By Chevalley's theorem, there exists a point $P \in X_f(\F_q)$. After a linear change of coordinates on $\P^n$, we may assume $P = [0 : \ldots : 0 : 1]$. Such a change of coordinates may change whether a given hyperplane $H$ is a component of $X_f$, or the type of $X_f \cap H$, but the hypothesis still holds, and $f$ has type $i=1,2,3$ if and only if it does after the change of coordinates.

    Let $H_i \colon x_i = 0$ denote a coordinate hyperplane and let $f_i = f(x_0, \ldots, x_{i-1}, 0, x_{i+1}, \ldots, x_n)$ be the (possibly zero) cubic form cutting.out the intersection $X_{f_i} = X_f \cap H_i \subseteq H_i \simeq \P^{n-1}$. Note that $H_i$ is an irreducible component of $X_f$ if and only if $f_i = 0$, or equivalently $x_i \mid f$. If $f_i \neq 0$ then $X_{f_i}$ is a configuration of conjugate hyperplanes, i.e. $f_i$ has factorization type 1, 2, or 3 from Definition \ref{def:types}.

    We first consider the case where $f_i = 0$ for some $i$. Since $n \geq 3$, there is at least one $j$ such that $f_j \neq 0$, hence we have $x_i \mid f_j$. Since $f_j$ is the product of conjugate linear forms, we must have $f_j = a_{iii}x_i^3$ (recall $a_{ijk}$ is the coefficient of $x_ix_jx_k$ in $f$ as in \eqref{eq:cubic_form}). In particular, this implies $x_{j'} \nmid f$ for all $j' \neq i,j$, and for any such $j'$ we have $f_{j'} = a_{iii}x_i^3$. From this we deduce
    \[f = a_{iii}x_i^3 + x_ix_jg_j = a_{iii}x_i^3 + x_ix_{j'}g_{j'}\]
    for linear forms $g_j, g_{j'} \in \F_q[x_0, \ldots, x_n]$. This forces $f = a_{iii}x_i^3 + a_{ijj'}x_ix_jx_{j'}$. Then there exists $k \notin \{i,j,j'\}$ such that $f = f_k\neq 0$ has factorization type 1, 2, or 3.
    
    Suppose now that $f_i \neq 0$ for all $i$. For $0 \leq i < n$ we have $f_i = \prod_{\sigma \in \Gal(\F_{q^3}/\F_q)} \sigma(\ell_i)$ for a nonzero linear form $\ell_i \in \F_{q^3}[x_0,  \ldots, x_n]$. Since $P \in X_{f_i}$ we  must have $\ell_i(0, \ldots, 0, 1) = 0$ and thus the $x_n$-coefficient of $\ell_i$ vanishes. Using this, we claim that $a_{ijn} = 0$ for all $i,j \in \{0, \ldots, n\}$. For any $i,j$, there exists $k \in \{0,\ldots, n\} - \{i,j,n\}$ and the $x_ix_jx_n$ terms of $f$ and $f_k$ coincide. But $f_k = \prod_{\sigma \in \Gal(\F_{q^3}/\F_q)} \sigma(\ell_k)$ for $\ell_k$ with zero $x_n$-coefficient, and thus $a_{ijn} = 0$. We conclude that no terms containing $x_n$ show up in $f$, and therefore $f = f_n$ has type 1, 2, or 3, as desired.   
\end{proof}

\begin{remark}
    Lemma \ref{lem:inductive step} fails to hold when $n = 2$, as $f = x_0x_1x_2$ is readily seen to be a counterexample. There also exist smooth and geometrically irreducible plane curves over $\F_q$ whose coordinate cubic forms are all type 1 or 2 (type 3 cannot occur for binary cubic forms). For example, when $3 \nmid q$ this is the case for diagonal plane cubics of the form $X \colon c_0 x_0^3 + c_1 x_1^3 + c_2 x_2^3 = 0$ if $c_0, c_1, c_2$ represent distinct classes in $\F_q^\times / (\F_q^\times)^3$.
\end{remark}

\begin{proof}[Proof of Proposition \ref{prop:lifting_prob_1}]
	For $n=1$, the case of primitive binary cubic forms $f \in \Z_p[x_0, x_1]$, we observe that factorization types 1 and 2 correspond to to $\overline{f}$ having a linear factor of multiplicity 3 and $\overline{f}$ being irreducible over $\F_p$, respectively. If $\overline{f}$ has neither of these factorization types, then it has a linear factor of multiplicity 1 and so must $f$ by Hensel's lemma.
	
	For $n \geq 2$ we proceed by induction on $n$. When $n=2$, the case of plane cubic curves, if $\overline{X_f}$ is smooth and geometrically irreducible then by an application of the Hasse--Weil bound on $\overline{X_f}$ we have
    \[\#\overline{X_f}(\F_p) \geq p + 1 - 2\sqrt{p} > 0.\]
    In particular, there exists a smooth $\F_p$-point to lift via Hensel's lemma to a point in $X_f(\Q_p)$. If $\overline{X_f}$ is geometrically irreducible but not smooth, its normalization $\widetilde{\overline{X_f}}$ is a smooth genus zero curve over $\F_p$ with $p+1$ $\F_p$-points. At most 2 of those points map to singular points of $\overline{X_f}$, so there are at least $p-1 > 0$ smooth $\F_p$-points on $\overline{X_f}$ which may be lifted to $X_f(\Q_p)$ via Hensel's Lemma.

    This leaves only the geometrically reducible cases. When $\overline{X_f}$ has a reduced line defined over $\F_p$ as a component, it contains $p+1$ $\F_p$-points; again, at most two of these intersect another component, leaving at least $p-1 > 0$ smooth points which may be lifted via Hensel's lemma. The only remaining possibilities lacking a reduced line defined over $\F_p$ as an irreducible component are precisely types 1, 2, and 3 of Definition \ref{def:types}, concluding the proof for $X_f$ a plane curve. 

    Suppose now that $n > 2$ and assume the conclusion holds for $n-1$. Since $X_f(\Q_p) = \emptyset$ then for all hyperplanes $H \subset \P^n$ defined over $\Q_p$ we have $(X_f \cap H)(\Q_p) = \emptyset$. Fixing a model for $H$ over $\Z_p$ and reducing modulo $p$, we have either $\overline{(X_f \cap H)} = \overline{H}$ or $\overline{(X_f \cap H)} \subsetneq \overline{H} \simeq \P^{n-1}_{\F_q}$ is a union of conjugate hyperplanes defined over $\F_{p^3}$ by the inductive hypothesis. Invoking Lemma \ref{lem:inductive step}, $\overline{f}$ must have type 1, 2, or 3.    
\end{proof}

\begin{remark}
    Kopparty and Yekhanin \cite[Lemma 3.2]{KoppartyYekhanin} show that for $d$ prime, $n \geq \frac{d}{2} + 1$, and $q \geq 32n^4$, if a degree $d$ form $f \in \F_q[x_0, \ldots, x_n]$ has no $\F_q$-solutions (other than $x_0 = \ldots = x_n = 0$), then there exists a linear form $\ell$ such that $f = \prod_{\sigma \in \Gal(\F_{q^d}/\F_q)} \sigma \ell$. Their proof invokes the Chevalley--Warning theorem as well as the Weil conjectures. 
\end{remark}

An immediate consequence of Proposition \ref{prop:lifting_prob_1} is that if the reduction $\overline{f}$ of a primitive cubic form $f \in \Z_p[x_0, \ldots, x_n]$ does not have type 1, 2, or 3, then $X_f(\Q_p) \neq \emptyset$. In other words, the probability of $X_f$ having a $\Q_p$ point, given this condition on $\overline{f}$, is 1:
\[\frac{\mu_p \left( \left\{f \in \Z_p[x_0, \ldots, x_n] : f \text{ primitive cubic form},\ \overline{f} \text{ not type } 1,\ 2,\ 3,\ \text{and } X_f(\Q_p) \neq \emptyset \right\}\right)}{\mu_p \left( \left\{f \in \Z_p[x_0, \ldots, x_n] : f \text{ primitive cubic form},\ \overline{f} \text{ not type } 1,\ 2,\ 3 \right\}\right)} = 1,\]
where $\mu_p$ denotes the normalized $p$-adic Haar measure as usual.

\section{Lifting probabilities that are not 1}
\label{sec:lifting_not_1}

Let us return to the challenge of computing $\rho_{3,n}(p)$, the probability that a nonzero cubic form over $\Z_p$ in $n+1$ variables has a solution, which we denote hereafter by $\rho_n(p)$, suppressing the degree, to streamline the notation. In this section we drop the dependence on $p$ for our notation, writing $\rho_n$ for the local probability $\rho_n(p)$. 

In the previous section, we determined that when $\overline{f}$ does not have type 1, 2, or 3, $X_f(\Q_p)$ is nonempty. The goal of this section is to complete the computation of $\rho_n(p)$, beginning with defining lifting probabilities conditional on $\overline{f}$ having type $i=1, 2, 3$.

\begin{definition}\label{def:sigma}
Let $\sigma_{n,i}$ denote the probability (in the sense of normalized Haar measure, as usual) that a primitive cubic form $f$ has a $\Q_p$-solution, given that $\overline{f}$ has factorization type $i$, 
\[\sigma_{n,i} = \frac{\mu_p \left( \{  f \in \Z_p[x_0, \ldots, x_n] : f \text{ primitive cubic form},\ \overline{f} \text{ has type } i,\ X_f(\Q_p) \neq \emptyset\} \right)}{\mu_p \left( \{ f \in \Z_p[x_0, \ldots, x_n] : f \text{ primitive cubic form},\ \overline{f} \text{ has type } i\} \right)}.\]
\end{definition}

By Proposition \ref{prop:lifting_prob_1} and the definition of $\xi_{n,i}$, we have
\begin{equation}\label{eq:rho_structure_eq}
    \rho_n = \xi_{n,0} + \xi_{n,1}\sigma_{n,1} + \xi_{n,2}\sigma_{n,2} + \xi_{n,3}\sigma_{n,3}.
\end{equation}
By reducing the problem to computing $\sigma_{n,i}$, we obtain some additional structure in being able to write our cubic forms as the product of conjugate linear forms. Namely, we will make key changes of variables that allow $\sigma_{n,i}$ to be related to new conditional lifting probabilities of cubic forms in fewer variables. This process will then be iterated until we have enough relations to solve for all the lifting probabilities.

Before embarking on this journey, let us define a few additional probabilities. Let $\rho_{n}^{(j)}$ denote the probability that a nonzero cubic form over $\Z_p$ in $n+1$ variables satisfying condition $(j)$ has a $\Q_p$-solution. We then have the analogue of \eqref{eq:rho_structure_eq},
\begin{equation}\label{eq:rho_(j)}
    \rho_n^{(j)} = \xival[j]{n}{0} + \sum_{i=1,2,3} \xival[j]{n}{i} \sigma_{n,i}.
\end{equation}

\begin{remark}\label{rem:compare_to_BCF}
In the $n=2$ case, Bhargava, Cremona, and Fisher \cite{BCF_plane_cubic} determine
\[\rho_2 = 1 - \frac{p^9 - p^8 + p^6 - p^4 + p^3 + p^2 - 2p + 1}{3(p^2 + 1)(p^4 + 1)(p^6 + p^3 + 1)}, \text{ so } 1- \rho_2 \sim \frac{1}{3p^3}.\]
Their approach takes advantage of the fact that plane curves reducing to a ``triangle" configuration, i.e.\ when $\overline{f}$ has factorization type 3, have no solutions, i.e.\ $\sigma_{2,3} = 0$. While they employ similar changes of variables and overall strategy, this allows for a considerably simpler web of relations. Our approach is not hindered by allowing $n=2$, so we will not exclude this case; the reader may view our results in the $n=2$ case as a repackaging of their result.

There are some differences between the two papers worth briefly highlighting. One minor difference comes from our decision for $\xi_{n,i}$ to be the probability that a \textit{nonzero} form up to scaling possesses factorization type $i$. Others appear to be merely notational (e.g.\ they denote with $\beta_{1}''$ what we denote by $\xi_{2,2}^{(1)}$), but we find substantial utility in being able to easily index both the factorization types $i=1,2,3$ and the conditions $(j)$. 
\end{remark}

\subsection*{Valuation tables}

Here we introduce a bookkeeping tool that will be employed extensively in the remainder of this section, which we refer to as a \textit{valuation table} for a $p$-adic cubic form $f \in \Z_p[x_0, \ldots, x_n]$. This table records known information about the $p$-adic valuations of the coefficients of $f$, streamlining how we keep track of the recursive process through which we are able to compute $\rho_n, \rho_n^{(j)}, \sigma_{n,i}, \sigma_{n,i}^{(j)}$, and other lifting probabilities we will soon define. 

Similar constructions appeared in the calculation of $\rho_2$ \cite[e.g.\ Lemma 12]{BCF_plane_cubic}, as well as in that of the density of soluble quadratic forms \cite[Lemma 3.3]{BCFJK}. To handle cubics in more than three variables while keeping things from getting (too) out of hand, we employ a \textit{blocking} strategy which allows our arguments to be essentially uniform in $n$.

Suppose we partition the set of variables $\{x_0, \ldots, x_n\}$ into $r$ (nonempty) subsets. Since for our purposes we only need $r \leq 4$, we label them by $S, T, U, W$ to avoid excessive decoration. Generally, each subset will consist of consecutively indexed variables, e.g.\ $S = \{x_0, \ldots, x_{i-1}\}$ and $T= \{x_i, \ldots, x_n\}$. 

This induces a partition of the set of cubic monomials $\{x_i x_j x_k : 0 \leq i \leq j \leq k \leq n\}$ into subsets of the form
\[S^sT^tU^uW^w = S^s \times T^t \times U^u \times W^w \text{ such that } s+t+u+w = 3,\]
that is, those monomials $x_ix_jx_k$ for which exactly $s$ of $x_i, x_j, x_k$ lies in $S$, exactly $t$ lie in $T$, etc. When $r < 4$, we drop the extraneous empty subset $W$, and possibly also $U$, from the notation altogether.

\begin{definition}
Let
\[v_{stuw} = \min \{v_p(a_{ijk}) : x_ix_jx_k \in S^sT^tU^uW^w \}.\]
A \textbf{valuation table for $\boldsymbol{f}$} is a collection of arrays, indexed by $w$, that records information about $v_{stuw}$ in the $u$-th row and $t$-th column.
\end{definition}

This is perhaps best seen by illustration for $r=2,3,4$ separately.

\begin{example}
    Suppose $r=2$, so $\{x_0, \ldots, x_n\} = S \coprod T$. Let
    \[v_{st} = \min\{v_p(a_{ijk}) : x_ix_jx_k \in S^sT^t\}.\]
    In this case, the valuation table is a single row, with columns indexed by the value of $t$ starting with $t=0$ on the left and $t=3$ on the right.
    \begin{table}[h]
    \centering
    \caption{The shape of a valuation table for a partition with $r=2$}
    \label{tab:sample_val_table_r=2}
    \begin{tabular}{c c c c}
         $*$ & $*$ & $*$ & $*$\\
    \end{tabular}
\end{table}
\end{example}

\begin{example}
Suppose $r=3$, so $\{x_0, \ldots, x_n\} = S \coprod T \coprod U$. Let
\[v_{stu} = \min\{v_p(a_{ijk}) : x_ix_jx_k \in S^sT^tU^u\}.\]
Now the valuation table forms a triangle, with the entry in the $u$-th row and $t$-th column (indexing by zero, so the top left entry is $u=t=0$, the bottom left entry is $(u,t) = (3,0)$, the top right is $(u,t) = (0,3)$, etc.) corresponds to information known about $v_{stu}$ with $s=3-t-u$.
    \begin{table}[h]
    \centering
    \caption{The shape of a valuation table for a partition with $r=3$}
    \label{tab:sample_val_table_r=3}
    \begin{tabular}{c c c c}
         $*$ & $*$ & $*$ & $*$\\
         $*$ & $*$ & $*$ & \\
         $*$ & $*$ & &\\
         $*$ &&&\\
    \end{tabular}
\end{table}
\end{example}

\begin{example}
Suppose $r=4$, so $\{x_0, \ldots, x_n\} = S \coprod T \coprod U \coprod W$. Let
\[v_{stuw} = \min\{v_p(a_{ijk}) : x_ix_jx_k \in S^sT^tU^uW^w\}.\]
We would like to represent information about $v_{stuw}$ along lattice points forming a tetrahedron, but we will settle for ``slicing" the tetrahedron at each value of $w$. 

That is, the information about $v_{stuw}$ is recorded in the $u$-th row and $t$-th column of the $w$-th slice (again indexing by zero as in the previous example), with $s=3-t-u-w$, of the table below.
    \begin{table}[h]
    \centering
    \caption{The shape of a valuation table for a partition with $r=4$}
    \label{tab:sample_val_table_r=4}
    \begin{tabular}{c c c c | c c c | c c | c}
        \multicolumn{4}{c|}{$w=0$} & 
        \multicolumn{3}{c|}{$w=1$} &
        \multicolumn{2}{c|}{$w=2$} & $w=3$ \\ \hline 
         $*$ & $*$ & $*$ & $*$ & $*$ & $*$ & $*$ & $*$ & $*$ & $*$\\
         $*$ & $*$ & $*$ &     & $*$ & $*$ &     & $*$ &     & \\
         $*$ & $*$ &     &     & $*$ &     &     &     &     & \\
         $*$ &     &     &     &     &     &     &     &     & \\
    \end{tabular}
\end{table}
\end{example}

A \textit{corner} entry in a valuation table corresponds to (exactly) one of $s,t,u,w$ equal to 3. For example, if $s=3$, the polynomial $h_S = \sum_{x_i, x_j, x_k \in S} a_{ijk}x_ix_jx_k$ corresponds to the cubic form in $\#S$ variables obtained from $f$ by setting $x_i=0$ for all $x_i \in T,U,W$. The other corner entries similarly represent cubic forms in only the variables in $T,U,W$, etc.\ Note that when $r=4$, the single entry in the $w=3$ slice is a corner, but none of the entries in the $w=1,2$ slices are corners.

We will occasionally decorate the corner entries to keep track of additional information. A subscript $*_i$ for $i=1,2,3$ indicates \textit{reduction type} $i$: that is, if the $s=3$ corner entry is decorated as such, then the cubic form $h_S$ described above has reduction type $i$ after dividing out by any common factors of $p$.

A superscript $*^{(j)}$ for $j=1,2,3$ indicates \textit{condition $(j)$}: again using $s=3$, this indicates that $h_S$ above satisfies the point, line, or plane condition ($j=1,2,3$ respectively). Note that we are assuming an ordering on the variables within $S$ (or $T, U, W$) in declaring that $h_S$ satisfies condition $(j)$. We will always mean this to be with the variables in increasing order. For example, if $S=\{x_0, \ldots, x_{i-1}\},\ T = \{x_i, \ldots, x_n\}$, and the $t=3$ corner entry in a valuation table reads $=0^{(j)}$, it indicates that the cubic form
\[h_T(x_i, \ldots, x_n) = f(0, \ldots, 0, x_i, \ldots, x_n)\]
is nonvanishing modulo $p$ and satisfies condition $(j)$. That is, after setting all but its first $j$ variables $x_i, \ldots, x_{i+j-1}$ to zero, we only have the trivial solution modulo $p$,
\[\overline{h_T}(x_i, \ldots, x_{i+j-1}, 0, \ldots, 0) =0 \implies x_i = \ldots = x_{i+j-1} = 0.\]

\subsection{Phase I}

We first set out to compute $\sigma_{n,i}$. In this first phase of computation, we will employ key changes of variables which will become a recurring theme in later phases. The idea is that when $f$ has factorization type $i$, we can produce another $p$-adic cubic form $\fI$ such that $X_f(\Q_p)$ and $X_{\fI}(\Q_p)$ are in bijection, with $\fI$ taking a shape particularly amenable to further analysis.

Suppose $f \in \Z_p[x_0, \ldots, x_n]$ is a cubic form with factorization type $i$. After scaling, we may assume $f$ is primitive, i.e.\ $p$ does not divide all coefficients. Recall that $\overline{f} = \prod_{\sigma \in \Gal(\F_{p^3}/\F_p)} \sigma (b_0 x_0 + \ldots + b_n x_n)$ where $\dim \Span\{b_j\}_{j=0}^n = i$. After a linear change of the coordinates $x_j$, we may assume that $\{b_j\}_{j=0}^n$ is spanned by its first $i$ elements $b_0, \ldots, b_{i-1}$, and thus also that $b_j = 0$ for all $j \geq i$.

Translating things back to $f$, after the aforementioned change of coordinates we have
\[f = g(x_0, \ldots, x_{i-1}) + p\sum_{0 \leq j \leq k < i} \ell_{jk}(x_i, \ldots, x_n)x_jx_k + p\sum_{0 \leq j < i} q_j(x_i, \ldots, x_n)x_j + p h(x_i, \ldots, x_n)\]
where $\ell_{jk}$, $q_j$, and $h$ are linear, quadratic, and cubic forms in $n+1-i$ variables over $\Z_p$. This situation can be compactly described by taking $S=\{x_0, \ldots, x_{i-1}\}$ and $T=\{x_i, \ldots, x_n\}$ and forming the valuation table of $f$ below in Table \ref{tab:valuations_f}. 
\begin{table}[h]
\centering
    \caption{Valuation table for $f$}
    \label{tab:valuations_f}
    \begin{tabular}{c c c c}
         $=0_i$ & $\geq 1$ & $\geq 1$ & $\geq 1$  
    \end{tabular}
\end{table}

Suppose $[x_0 : \ldots : x_n] \in X_f(\Q_p)$. Throughout, after clearing denominators, we may assume $x_0, \ldots, x_n \in \Z_p$ and share no common factor of $p$. We claim that $p \mid x_0, \ldots, x_{i-1}$. Reducing modulo $p$, we have that $\overline{g}(x_0, \ldots, x_{i-1}) = 0$  is necessary. However, since $\overline{g} = \overline{f}$ has factorization type $i$, the only such solution is $x_0 = \ldots = x_{i-1} = 0 \pmod{p}$. Indeed,
\begin{itemize}
    \item when $i=1$ we have $g(x_0) = ax_0^3$ for some $a \in \Z_p^\times$;
    \item when $i=2$, $\overline{g}$ is an irreducible binary cubic form over $\F_p$ and hence has no nontrivial solutions;
    \item when $i=3$ we have $\overline{g}$ is a ternary cubic form in the ``triangle" configuration and hence has no nontrivial solutions.
\end{itemize}
This motivates another change of coordinates, this time replacing $x_0, \ldots, x_{i-1}$ by $px_0, \ldots, px_{i-1}$ and dividing by $p$,
\begin{align}
    \label{eq:fI} \fI &= \frac{1}{p} f(px_0,\ldots, px_{i-1}, x_i, \ldots, x_n) \\
    \nonumber &= p^2g(x_0, \ldots, x_{i-1}) + p^2\sum_{0 \leq j \leq k < i} \ell_{jk}(x_i, \ldots, x_n)x_jx_k + p\sum_{0 \leq j < i} q_j(x_i, \ldots, x_n)x_j + h_{\mathrm{I}}(x_i, \ldots, x_n),
\end{align}
and we have a bijection between $X_f(\Q_p)$ and $X_{\fI}(\Q_p)$. We pause to record how this change of variables affects the valuations of coefficients of $\fI$ below in Table \ref{tab:valuations_f_I}, opting to lean on this notation as things get more involved.
\begin{table}[h]
\centering
\caption{Valuation table for $\fI$}
\label{tab:valuations_f_I}
\begin{tabular}{r c c c c l}
    & $=2_i$ & $\geq 2$ & $\geq 1$ & $\geq 0$
\end{tabular}
\end{table}

We are now ready to define additional lifting probabilities and relate them to $\sigma_{n,i}$.

\begin{definition}\label{def:taus}
    Let $\tau_{n,ij}$ denote the probability that $f$ has a $p$-adic solution, given that $f$ has type $i$ and after the transformations described above, $\fI$ is primitive and has type $j$.
\end{definition}

\begin{definition}\label{def:sigma_prime}
    Let $\sigma_{n,i}'$ denote the probability that $f$ has a $p$-adic solution, given that $f$ has type $i$ and after the transformations described above, the coefficients of $\fI$ have minimal $p$-adic valuation at least 1.
\end{definition}

\begin{lemma}\label{lem:sigma} Let $n \geq 1$ and $i \in \{1,2,3\}$. If $i = n+1$ we have
\[\sigma_{1,2} = \sigma_{2,3} = 0.\]
For $i < n+1$, $\sigma_{n,i}$ satisfies the relation
   \begin{equation}\label{eq:sigmas}
    \sigma_{n,i} = \left( 1 - \frac{1}{p^{\binom{n-i+3}{3}}}\right)\left(\xi_{n-i,0} + \sum_{1 \leq j \leq 3} \xi_{n-i,j} \tau_{n,ij} \right) + \frac{1}{p^{\binom{n-i+3}{3}}}\sigma_{n,i}'.  
\end{equation} 
\end{lemma}

\begin{proof}
    That $\sigma_{1,2} = \sigma_{2,3} = 0$ is merely the observation that if a binary (resp.\ ternary) cubic form $f$ has type 2 (resp.\ type 3), then its reduction contains no $\F_p$-points, and hence $X_f$ can have no $p$-adic points. %Moreover, the transformations described above produce $\fI = p^2 f$ in this case, which is uninteresting.

    Excluding $(n,i) = (1,2), (2,3)$ and given a polynomial $f$ with factorization type $i$, we perform the transformations above to obtain $\fI$ with valuation table given by Table \ref{tab:valuations_f_I}. Note that $\overline{\fI} = \overline{h_{\mathrm{I}}}$, so (provided $\overline{h_{\mathrm{I}}} \neq 0$) if $X_{f}(\Q_p) = X_{\fI}(\Q_p) = \emptyset$, then by Proposition \ref{prop:lifting_prob_1} we have $\overline{h_{\mathrm{I}}}$ has factorization type 1, 2, or 3.
    
    Since $h_{\mathrm{I}}$ is a general cubic form in $n-i+1$ variables, it has $\binom{n-i+3}{3}$ coefficients, hence the probability that $h_{\mathrm{I}}$ is primitive is given by $1-p^{\binom{n-i+3}{3}}$. If $h_{\mathrm{I}}$ is not primitive, then by definition the lifting probability is given by $\sigma_{n,i}'$. If $h_{\mathrm{I}}$ is primitive, then $\xi_{n-i,j}$ represents the probability that $h_{\mathrm{I}}$ has factorization type $j$ (recall $j=0$ means that $h_{\mathrm{I}}$ does not have types 1, 2, or 3). For $j\neq 0$, we again find by definition that $\tau_{n,ij}$ is the probability of a lift.

    In the $j=0$ case, $\overline{h_{\mathrm{I}}}$ does not have type 1, 2, or 3, so it has a nontrivial solution $h_{\mathrm{I}}(x_i,\ldots,x_n) \equiv 0 \pmod{p}$ which lifts to a $p$-adic solution. From this it follows that $\fI(0, \ldots, 0, x_i,\ldots,x_n) \equiv 0 \pmod{p}$ lifts to a $p$-adic point on $X_{\fI}$. Thus in this case the lifting probability is 1. Putting this together with our previous observations yields \eqref{eq:sigmas}.
\end{proof}

\begin{remark}\label{rem:sigma_edge_cases}
   We should point out that $\tau_{n,ij}$ is not well defined when $i+j > n+1$, since the polynomial $h_{\mathrm{I}}$ in $n-i+1$ variables cannot have factorization type $j > n-i+1$. However, in these cases we have $\xi_{n-i,j} = 0$, so \eqref{eq:sigmas} is still well defined. We will therefore sweep this abuse of notation under the rug, both here and in similar future instances.
\end{remark}

It will be useful to define additional probabilities contingent on whether or not $h_{\mathrm{I}}$ satisfies the point, line, or plane condition. 

\begin{definition}
    Let $\sigma_{n,i}^{(k)}$ denote the probability that $f$ has a $p$-adic solution, given that $f$ has type $i$ and after the transformations described above, $\fI$ satisfies condition $(k)$, i.e. $\overline{\fI}(x_i, \ldots, x_{i+k-1}, 0, \ldots, 0)$ has no nontrivial solutions.
\end{definition}
Thus $\sigma_{n,i}^{(k)}$ describes the lifting probability when $f$ has type $i$ and $\fI$ has the valuation table given in Table \ref{tab:valuations_f_I_k} where $h_{\mathrm{I}}$ satisfies condition $(k)$. 
\begin{table}[h]
\centering
\caption{Valuation table for $\fI$ where $\overline{\fI}$ satisfies condition $(k)$}
\label{tab:valuations_f_I_k} 
\begin{tabular}{r c c c c l}
    & $=2_i$ & $\geq 2$ & $\geq 1$ & $= 0^{(k)}$
\end{tabular}
\end{table}

Next we give an analogue of Lemma \ref{lem:sigma}. The proof is entirely analogous --- in fact it is slightly simpler in that $\fI$ is known to be primitive since it satisfies condition $(k)$ --- so we omit it.

\begin{lemma}\label{lem:sigma_cond} Let $n \geq 2$ and $1 \leq i,k \leq 3$. Then $\sigma_{n,i}^{(k)}$ satisfies the relation
\begin{equation}\label{eq:sigma_cond}
    \sigma_{n,i}^{(k)} = \xi_{n-i,0}^{(k)} +  \sum_{1 \leq j \leq 3} \xi_{n-i,j}^{(k)} \tau_{n,ij}.
\end{equation}
\end{lemma}

We conclude this phase by computing a relation for $\sigma_{n,i}'$. This offers a first glimpse into the interplay between the various lifting probabilities we have defined so far, and a glimmer of hope that in continuing on to phases II and III, we will actually accumulate enough relations to solve for $\rho_n$.

\begin{lemma}\label{lem:sigma_prime}
Let $n \geq 2$ and $1 \leq i \leq 3$ and $(n,i) \neq (2,3)$. Then $\sigma_{n,i}'$ satisfies the relation
\begin{equation}\label{eq:sigma_prime}
    \sigma_{n,i}' = 1 - \frac{1}{p^{i\binom{n-i+2}{2}}} + \frac{1}{p^{i\binom{n-i+2}{2}}} \left( \left(1 - \frac{1}{p^{\binom{n-i+3}{3}}}\right) \left( \xi_{n-i,0} +  \sum_{1 \leq j \leq 3} \xi_{n-i,j}\sigma_{n,j}^{(i)} \right) + \frac{1}{p^{\binom{n-i+3}{3}}} \rho_n^{(i)} \right).
\end{equation}
\end{lemma}

\begin{proof}
    The starting point for $\sigma_{n,i}'$ is a cubic form $f$ from which $\fI$ is produced via the phase I transformations \eqref{eq:fI}. If $h_{\mathrm{I}}$ is not primitive, then the valuation table of $\frac{1}{p} \fI$ is given below.
\begin{table}[h]
\centering
\caption{Valuation table for $\frac1p \fI$}
\label{tab:valuations_f_I/p}
\begin{tabular}{r c c c c l}
    & $=1_i$ & $\geq 1$ & $\geq 0$ & $\geq 0$
\end{tabular}
\end{table}

    Reducing modulo $p$, we have
    \[\overline{\frac1p \fI} = \sum_{0 \leq j < i} \overline{q_j}(x_i, \ldots, x_n)x_j + \overline{\frac1p h_{\mathrm{I}}}(x_i, \ldots, x_n).\]
    Suppose for some $0 \leq j < i$ the quadratic form $q_j$ is nonvanishing mod $p$. Then there exists some (nontrivial) input $(x_i, \ldots, x_n)$ for which $q_j(x_i, \ldots, x_n) \not\equiv 0 \pmod{p}$. Then $\overline{\frac1p \fI}$ is linear in $x_j$, so it has a solution which lifts to a $p$-adic solution, producing a point in $X_f(\Q_p)$. 

    Suppose instead that all $q_j$ fail to be primitive. Counting coefficients of each $q_j$, this happens with probability $1/p^{i\binom{n-i+2}{2}}$. Now we turn our attention to the factorization type of $h_{\mathrm{I}}/p$. As a cubic form in $n-i+3$ variables, it fails to be primitive with probability $1/p^{\binom{n-i+3}{3}}$, in which case the valuation table for $\fI/p$ is given by Table \ref{tab:valuations_f_I/p_v2} below.
\begin{table}[h]
\centering
\caption{Valuation table for $\frac1p \fI$}
\label{tab:valuations_f_I/p_v2}
\begin{tabular}{r c c c c l}
    & $=1_i$ & $\geq 1$ & $\geq 1$ & $\geq 1$
\end{tabular}
\end{table}

    Here we can divide again by $p$, returning us to the case of a general cubic form in $n+1$ variables satisfying \textit{condition} $(i)$, since $\overline{\frac1{p^2}\fI}(x_0, \ldots, x_{i-1}, 0, \ldots, 0)$ has no solutions. Thus in the case that $h_{\mathrm{I}}/p$ fails to be primitive, the lifting probability is given by $\rho_n^{(i)}$.

    If $h_{\mathrm{I}}/p$ is primitive (which occurs with probability $1 -1/p^{\binom{n-i+3}{3}}$, then its reduction has factorization type $j=1,2,3$ with probability $\xi_{n-i,j}$, and none of those with probability $\xi_{n-i,0}$. In the latter case, a liftable solution exists by the same argument as in Lemma \ref{lem:sigma} using Proposition \ref{prop:lifting_prob_1}.

    In the former case, we claim the probability of a liftable solution is given by $\sigma_{n,j}^{(i)}$. To realize this, we first carry out a change of coordinates in $x_i, \ldots, x_n$ isolating the first $j$ variables of $h_{\mathrm{I}}/p$. Repartitioning the set of variables into $S = \{x_0, \ldots, x_{i-1}\}$, $T = \{x_{i}, \ldots, x_{j-1}\}$, and $U = \{x_j, \ldots, x_n\}$, we obtain the left side of the valuation table below suggesting a natural reindexing of the variables to form the table on the right side.
    \begin{table}[h]
\centering
\caption{Valuation tables for $\frac1p \fI$ after isolating $x_i, \ldots, x_{i+j-1}$ in $h_{\mathrm{I}}/p$ and after renumbering}
\label{tab:valuations_f_I/p_v3}
\begin{tabular}{r c c c c l c r c c c c l}
    $S$ & & & & & $T$ & & $T$ & & & & & $S \cup U$\\
    & $=1_i$ & $\geq 1$ & $\geq 1$ & $=0_j$ & & & & $=0_j$ & $\geq 1$ & $\geq 1$ & $\geq 1^{(i)}$\\
    & $\geq 1$ & $\geq 1$ & $\geq 1$ & & & $\xrightarrow{\text{reindex}}$ \\
    & $\geq 1$ & $\geq 1$ \\
    & $\geq 1$ \\
    $U$ 
\end{tabular}
\end{table}

From Table \ref{tab:valuations_f_I/p_v3} we see that after reindexing and performing the usual phase I transformations \eqref{eq:fI} on the resulting cubic form, we land precisely in the situation of $\sigma_{n,j}^{(i)}$. Putting together these observations yields \eqref{eq:sigma_prime}, completing the proof of the lemma.
\end{proof}

\begin{example}[binary cubic forms]
    We can illustrate the approach of phase I by computing $\rho_1$, the density of $p$-adic binary cubic forms with at least one root, which can be deduced from the existing literature; see e.g.\ \cite[\S 1.2.3]{BCFG}.

    We have $\xi_{1,3} = \xival[1]{1}{3} = 0$ by Lemmas \ref{lem:xi_no_cond} and \ref{lem:xi_cond} and $\sigma_{1,2} = 0$ by Lemma \ref{lem:sigma}. Thus the equations
    \eqref{eq:rho_structure_eq} for $\rho_n$ and \eqref{eq:rho_(j)} for $\rho_n^{(1)}$ specialize to
    \begin{alignat*}{2}
        \rho_1 &= \xi_{1,0} + \xi_{1,1}\sigma_{1,1} &&= \frac{2p(p+1)}{3(p^2 + 1)} + \frac1{p^2 + 1}\sigma_{1,1} \\
        \rho_1^{(1)} &= \xival[1]{1}{0} + \xival[1]{1}{1}\sigma_{1,1} &&= \frac{2(p^2 - 1)}{3p^2} + \frac1{p^2}\sigma_{1,1}.
    \end{alignat*}
    By Lemmas \ref{lem:sigma}, \ref{lem:sigma_cond}, \ref{lem:sigma_prime}, specializing \eqref{eq:sigmas}, \eqref{eq:sigma_cond}, \eqref{eq:sigma_prime} to $n=1$, $i=1$ reveals
    \begin{alignat*}{2}
        \sigma_{1,1} &= \left(1 -\frac1p\right)\tau_{1,11} + \frac1p \sigma_{1,1}' &&= \frac1p \sigma_{1,1}'\\
        \sigma_{1,1}^{(1)} &= \tau_{1, 11} &&= 0 \\
        \sigma_{1,1}' &= 1 - \frac1p + \frac1p\left( \left(1 - \frac1p\right)\sigma_{1,1}^{(1)} + \frac1p \rho_1^{(1)}\right) &&= 1 - \frac1p + \frac1{p^2} \rho_1^{(1)},
    \end{alignat*}
    with the right hand equality following once we observe that $\tau_{1, 11} = 0$. 
   
    We can see this directly by unwinding Definition \ref{def:taus} as follows. Suppose $f$ is a binary form with type 1, so after a change of variables we may assume its valuation table is given by Table \ref{tab:valuations_f} with $S = \{x_0\}$ and $T = \{x_1\}$. Any primitive solution $[x_0 : x_1]$ therefore has $p \mid x_0$. After the transformation \eqref{eq:fI}, we see $\fI$ has valuation table given below, and thus $p \mid x_1$.
    \begin{table}[h]
    \centering
    \label{tab:example_n=1}
    \caption{Valuation table for $\fI$ when $n=1$ and $i=1$}
    \begin{tabular}{r c c c c l}
    & $=2_1$ & $\geq 2$ & $\geq 1$ & $= 0_1$
    \end{tabular}
    \end{table}
    This contradicts the original $[x_0 : x_1]$ being primitive, so $\tau_{1, 11}=0$.    Solving the relations above among $\rho_1$, $\rho_1^{(1)}$, $\sigma_{1,1}$, and $\sigma_{1,1}'$, we obtain
    \begin{equation}\label{eq:rho1}
        \rho_1 = \frac{2p^4 + 3p^3 + p^2 + 3p + 2}{3(p^4 + p^3 + p^2 + p + 1)} = 1 - \frac{(p^2 + 1)^2}{3(p^4 + p^3 + p^2 + p + 1)}.
    \end{equation} 
    
\end{example}

\subsection{Phase II}

The goal of phase II is to compute $\tau_{n,ij}$: the lifting probability given that $f$ has type $i$ and the $\fI$ produced in phase I has type $j$. Here $i,j \in \{1,2,3\}$ and $i+j \leq n+1$ (see Remark \ref{rem:sigma_edge_cases}). Our approach is essentially the same as in phase I, in that we first isolate the first $j$ variables of $h_{\mathrm{I}}$, then produce $\fII$ such that $X_{\fII}(\Q_p), X_{\fI}(\Q_p), X_f(\Q_p)$ are in bijection, and analyze the reduction $\overline{\fII}$.

Recall from phase I that $\overline{\fI} = \overline{h_{\mathrm{I}}}$ for $h_{\mathrm{I}}(x_i, \ldots, x_n)$ with factorization type $j$. After an invertible linear change of coordinates in \textit{only} $x_i, \ldots, x_n$, we may assume $\fI$ has the valuation table below for the partition $S = \{x_0, \ldots, x_{i-1}\}$, $T=\{x_i, \ldots, x_{i+j-1}\}$, and $U = \{x_{i+j}, \ldots, x_n\}$ (note that $U = \emptyset$ if $i+j = n+1$). 

\begin{table}[h]
    \centering
    \caption{Valuation table for $\fI$}
    \label{tab:val_table_fI_after_iso}
    \begin{tabular}{c c c c}
         $=2_i$ & $\geq 2$ & $\geq 1$ & $=0_j$\\
         $\geq 2$ & $\geq 1$ & $\geq 1$ & \\
         $\geq 1$ & $\geq 1$ & &\\
         $\geq 1$ &&&\\
    \end{tabular}
\end{table}

Note that $[x_0 : \ldots : x_n] \in X_{\fI}(\Q_p)$  implies $p \mid x_i, \ldots, x_{i+j-1}$. This leads us to define
\[\fII = \frac1p \fI(x_0, \ldots, x_{i-1}, px_i, \ldots, px_{i+j-1}, x_{i+j}, \ldots, x_n)\]
with valuation table below such that $X_{\fII}(\Q_p)$ and $X_{\fI}(\Q_p)$ are in bijection.

\begin{table}[h]
    \centering
    \caption{Valuation table for $\fII$}
    \label{tab:val_table_fII}
    \begin{tabular}{c c c c}
         $=1_i$ & $\geq 2$ & $\geq 2$ & $=2_j$\\
         $\geq 1$ & $\geq 1$ & $\geq 2$ & \\
         $\geq 0$ & $\geq 1$ & &\\
         $\geq 0$ &&&\\
    \end{tabular}
\end{table}

Note the $SU^2$ entry in Table \ref{tab:val_table_fII} above. If these coefficients are nonvanishing modulo $p$, then $\overline{\fII}$ is \textit{linear} in at least one of the variables $x_0, \ldots, x_{i-1}$, so an $\F_p$-solution can be found and lifted to a $\Q_p$-point on $X_{\fII}$. If the $SU^2$ entry has valuation at least 1, then we must analyze $\overline{\fII} = h_{\mathrm{II}}(x_{i+j}, \ldots, x_n)$.

\begin{definition}
    Let $\theta_{n,ijk}$ denote the probability that $f$ has a $p$-adic solution given that
    \begin{itemize}
        \item $f$ has type $i$, 
        \item $\fI$ has type $j$, 
        \item the $SU^2$ coefficients of $\fII$ have minimal valuation at least 1, and
        \item $\fII$ is primitive and reduces to type $k$.
    \end{itemize}
\end{definition}

\begin{definition}
    Let $\tau_{n,ij}'$ denote the probability that $f$ has a $p$-adic solution given that $f$ has type $i$, $\fI$ has type $j$, and the resulting $\overline{\fII}$ vanishes modulo $p$.
\end{definition}

\begin{lemma}\label{lem:tau}
    Let $n \geq 1$ and $i,j \in \{1,2,3\}$. If $i+j = n+1$ then
    \[\tau_{n,ij} = 0.\]
    If $i+j < n+1$ then
    \begin{equation}\label{eq:tau}
        \tau_{n,ij} = \scalebox{0.85}{$\displaystyle\left(1-\frac{1}{p^{i\binom{n-i-j+2}{2}}}\right) + \frac{1}{p^{i\binom{n-i-j+2}{2}}}\left( \left(1- \frac{1}{p^{\binom{n-i-j+3}{3}}}\right)\left(\xi_{n-i-j,0} + \sum_{1 \leq k \leq 3} \xi_{n-i-j,k} \theta_{n,ijk} \right) + \frac{1}{p^{\binom{n-i-j+3}{3}}} \tau_{n,ij}' \right)$}.
    \end{equation}
\end{lemma}

\begin{proof}
    Suppose $i+j = n+1$, i.e. that $U$ is empty. We know from phase I that if $[x_0:\ldots:x_n] \in X_f(\Q_p)$ then $p \mid x_0, \ldots, x_{i-1}$. The reduction $\overline{\fII}$ has type $j$ in exactly $j$ variables, and hence $p \mid x_i, \ldots, x_n$ as well, a contradiction. Assume moving forward that $i+j < n+1$

    The probability that the coefficients of monomials in $SU^2$ are all divisible by $p$ is equal to that of $i$ independent \textit{quadratic} forms in $n-i-j$ variables vanishing modulo $p$, which is $1/p^{i\binom{n-i-j+2}{2}}$. If any of the $SU^2$ coefficients are $p$-adic units, then $\overline{\fII}$ is linear in $x_m$ for some $0 \leq m < i$, and there exists a point in $\overline{X_{\fII}}(\F_p)$ which lifts to $X_{\fII}(\Q_p)$ by Hensel's Lemma (this is the same argument used in Lemma \ref{lem:sigma_prime} with the $ST^2$ coefficients).

    If all the coefficients of $SU^2$ monomials are divisible by $p$, then the probability that $\overline{\fII}$ vanishes altogether is that of a cubic form in $n-i-j$ variables vanishing, $1/p^{\binom{n-i-j+3}{3}}$. In this case the lifting probability is $\tau_{n,ij}'$ by definition.

    Assuming $\overline{\fII} \neq 0$, we stratify by reduction type. For $k=1,2,3$, type $k$ occurs with probability $\xi_{n-i-j,k}$ and the associated lifting probability is $\theta_{n,ijk}$ by definition. If $\overline{\fII}$ has none of these types, which occurs with probability $\xi_{n-i-j,0}$ then $X_{\fII}(\Q_p) \neq \emptyset$ by Proposition \ref{prop:lifting_prob_1}.
\end{proof}

\begin{lemma}\label{lem:tau_prime}
    Let $n \geq 2$ and $i,j \in \{1,2,3\}$ such that $i+j < n+1$. Then we have
    \begin{equation}\label{eq:tau_prime}
        \tau_{n,ij}' = 1 - \frac{1}{p^{ij(n-i-j+1) + j\binom{n-i-j+2}{2}}} + \frac{1}{p^{ij(n-i-j+1) + j\binom{n-i-j+2}{2}}}\left(\xi_{n-j,0}^{(i)} + \sum_{0 \leq k \leq 3} \xi_{n-j,k}^{(i)} \sigma_{n,k}^{(j)} \right)
\end{equation}
\end{lemma}

\begin{proof}
    We begin by dividing through by $p$, producing the valuation table below.

    \begin{table}[h]
    \centering
    \caption{Valuation table for initial step in computing $\tau_{n,ij}'$ (colors added for emphasis)}
    \label{tab:val_table_tau_prime_initial}
    \begin{tabular}{c c c c}
         $=0_i$ & $\geq 1$ & $\geq 1$ & $=1_j$\\
         $\geq 0$ & {\color{blue}$\geq 0$} & $\geq 1$ & \\
         $\geq 0$ & {\color{red}$\geq 0$} & &\\
         $\geq 0$ &&&\\
    \end{tabular}
\end{table}

    Looking at the center left column of Table \ref{tab:val_table_tau_prime_initial}, we argue that if the coefficients of monomials in $STU$ (shown in blue) or $TU^2$ (shown in red) are units in $\Z_p$, then we can find a lift. Let us see why explicitly.

    Suppose one of the $TU^2$ coefficients is nonzero and momentarily specialize $x_0 = \ldots = x_{i-1} = 0$. The resulting cubic form is linear in one of the variables in $T$, so there exists a Hensel-liftable solution to a point in $X_{\fII}(\Q_p)$ (we saw this previously in Lemmas \ref{lem:sigma_prime} and \ref{lem:tau}).

    Suppose instead that all $TU^2$ coefficients are divisible by $p$, but some $STU$ coefficient is a $p$-adic unit; for concreteness, say $a_{0in} \in \Z_p^\times$. Again we find $\overline{\frac1p \fII}$ is linear in $x_i$, so there exists an $\F_p$ point with, say $x_0 = 1,\ x_n = 1$, which can be lifted via Hensel's lemma in $x_i$ to a point in $X_{\fII}(\Q_p)$.

    Thus we have found a lift unless all $STU, TU^2$ coefficients vanish, which occurs with the same probability as $ij$ linear forms in $n-i-j$ variables and $j$ quadratic forms in $n-i-j$ variables simultaneously vanishing modulo $p$, $1/p^{ij(n-i-j+1) + j\binom{n-i-j+2}{2}}$. If this occurs, we repartition our set of variables into $(S \cup U) \coprod T$, as shown in Table \ref{tab:val_table_tau_prime_repartition}.

     \begin{table}[h]
    \centering
    \caption{Repartitioning step in computing $\tau_{n,ij}'$}
    \label{tab:val_table_tau_prime_repartition}
    \begin{tabular}{r c c c c l c r c c c c l}
        $S$ & &&&& $T$ && $S \cup U$ &&&&& $T$\\
          & $=0_i$ & $\geq 1$ & $\geq 1$ & $=1_j$ & & &  & $=0^{(i)}$ & $\geq 1$ & $\geq 1$ & $= 1_j$ &\\
         & $\geq 0$ & $\geq 1$ & $\geq 1$ & & & $\xrightarrow{\text{reindex}}$\\
         &$\geq 0$ & $\geq 1$ & &\\
         &$\geq 0$ &&&\\
         $U$
    \end{tabular}
\end{table}

The factorization type is now $k$ with probability $\xi_{n-j,k}^{(i)}$. If $k=0$ we have a lift by Proposition \ref{prop:lifting_prob_1}; otherwise, performing the usual phase I operations \eqref{eq:fI} to isolate the first $k$ variables of $S \cup U$ and relabel once more, we find ourselves in the situation of $\sigma_{n,k}^{(j)}$. Putting everything together yields \eqref{eq:tau_prime}.
\end{proof}

\subsection{Phase III}

Our final goal is to compute $\theta_{n,ijk}$ in terms of the various lifting probabilities already defined. We recall our initial situation from the definition: $f, \fI$, had reduction types $i,j$ and in phase II we found $\fII$ had reduction type $k$ with its $SU^2$ coefficients in $p\Z_p$; note this implies $i+j+k \leq n+1$. We record the valuation table below for this situation with $\{x_0, \ldots, x_n\} = S\coprod T \coprod U$ as before with $S=\{x_0, \ldots, x_{i-1}\}, T = \{x_i, \ldots, x_{i+j-1}\}$ as before and $U=\{x_{i+j}, \ldots, x_{n}\}$.

\begin{table}[h]
    \centering
    \caption{Valuation table for $\fII$}
    \label{tab:val_table_theta_init}
    \begin{tabular}{c c c c}
         $=1_i$ & $\geq 2$ & $\geq 2$ & $=2_j$\\
         $\geq 1$ & $\geq 1$ & $\geq 2$ & \\
         $\geq 1$ & $\geq 1$ & &\\
         $= 0_k$ &&&\\
    \end{tabular}
\end{table}

After possibly an invertible linear transformation over $\Z_p$ involving only $x_{i+j}, \ldots, x_n$, we may isolate the $k$ variables $x_{i+j}, \ldots, x_{i+j+k-1}$ in $\overline{\fII}$. To make this precise, we instead modify our partition by $U=\{x_{i+j}, \ldots, x_{i+j+k-1}\}$ and add the block $W = \{x_{i+j+k}, \ldots, x_n\}$. This is neatly represented in the valuation table below.

\begin{table}[h]
    \centering
    \caption{Valuation table for initial situation of $\theta_{n,ijk}$}
    \label{tab:val_table_theta_init_repart}
    \begin{tabular}{c c c c | c c c | c c | c}
        \multicolumn{4}{c|}{$w=0$} & 
        \multicolumn{3}{c|}{$w=1$} &
        \multicolumn{2}{c|}{$w=2$} & $w=3$ \\ \hline 
         $=1_i$ & $\geq 2$ & $\geq 2$ & $=2_j$ & $\geq 1$ & $\geq 1$ & $\geq 2$ & $\geq 1$ & $\geq 1$ & $\geq 1$\\
         $\geq 1$ & $\geq 1$ & $\geq 2$ &     & $\geq 1$ & $\geq 1$ &     & $\geq 1$ &     & \\
         $\geq 1$ & $\geq 1$ &     &     & $\geq 1$ &     &     &     &     & \\
         $=0_k$ &     &     &     &     &     &     &     &     & 
    \end{tabular}
\end{table}

As in phases I, II, we see that for any $[x_0 : \ldots : x_n] \in X_{\fII}(\Q_p)$, we must have $p \mid x_{i+j}, \ldots, x_{i+j+k-1}$, motivating yet another change of variables
\[\fIII = \frac1p\fII(x_0, \ldots, x_{i+j-1}, px_{i+j}, \ldots, px_{i+j+k-1}, x_{i+j+k}, \ldots, x_n)\]
with valuation table given below.
\begin{table}[h]
    \centering
    \caption{Valuation table for $\fIII$ (color added for emphasis)}
    \label{tab:val_table_fIII}
    \begin{tabular}{c c c c | c c c | c c | c}
        \multicolumn{4}{c|}{$w=0$} & 
        \multicolumn{3}{c|}{$w=1$} &
        \multicolumn{2}{c|}{$w=2$} & $w=3$ \\ \hline 
         $=0_i$ & $\geq 1$ & $\geq 1$ & $=1_j$ & $\geq 0$ & {\color{blue}$\geq 0$} & $\geq 1$ & $\geq 0$ & {\color{red}$\geq 0$} & $\geq 0$\\
         $\geq 1$ & $\geq 1$ & $\geq 2$ &     & $\geq 1$ & $\geq 1$ &     & $\geq 1$ &     & \\
         $\geq 2$ & $\geq 2$ &     &     & $\geq 2$ &     &     &     &     & \\
         $=2_k$ &     &     &     &     &     &     &     &     & 
    \end{tabular}
\end{table}

If the $STW$ (shown in blue) and $TW^2$ (shown in red) entries in Table \ref{tab:val_table_fIII} were both divisible by $p$ at least once, then this table would ``collapse" into one resembling Table \ref{tab:val_table_theta_init}. In making this precise, we prove the following.

\begin{lemma}\label{lem:thetas}
    Let $n \geq 2$ and $i,j,k \in \{1,2,3\}$. If $i+j+k = n+1$ then
    \[\theta_{n,ijk} = 0.\]
    If $i+j+k < n+1$ then we have
\begin{equation}\label{eq:thetas}
    \theta_{n,ijk} = \scalebox{0.9}{$\displaystyle 1- \frac{1}{p^{ij(n-i-j-k+1) + j\binom{n-i-j-k+2}{2}}} + \frac{1}{p^{ij(n-i-j-k+1) + j\binom{n-i-j-k+2}{2}}} \left(\xi_{n-j-k,0}^{(i)} + \sum_{1 \leq \ell \leq 3} \xi_{n-j-k,\ell}^{(i)} \theta_{n,jk\ell}\right)$}.
\end{equation}
\end{lemma}

\begin{proof}
    That $\theta_{n,ijk}=0$ when $i+j+k=n+1$ follows from assuming $[x_0: \ldots: x_n] \in X_{\fIII}(\Q_p)$ and tracing it back to a point on $X_f(\Q_p)$ with all coordinates divisible by $p$, producing a contradiction  as in Lemmas \ref{lem:sigma} and \ref{lem:tau}. Assume then that $i+j+k < n$.

    If the $STW$ or $TW^2$ coefficients (the blue and red entries, respectively in Table \ref{tab:val_table_fIII}) are in $\Z_p^\times$ then there exists an $\F_p$-point on $X_{\fIII}$ that lifts via Hensel's lemma. The argument follows exactly the same as that in Lemma \ref{lem:tau_prime}. These coefficients are all in $p\Z_p$ with probability equal to that of $ij$ linear forms in $n-i-j-k$ variables and $j$ quadratic forms in $n-k-j-k$ variables over $\F_p$ vanishing, given by $1/p^{ij(n-i-j-k+1) + j\binom{n-i-j-k+2}{2}}$.

    Assuming we are in this situation, the reduction is a cubic form in $n-j-k+1$ variables,
    \[\overline{\fIII} = \overline{\fIII}(x_0, \ldots, x_{i-1}, x_{i+j+k}, \ldots, x_n),\]
    due to all other monomials vanishing. Note that $\fIII$ is known to satisfy condition $(i)$. As is our custom, we analyze the lifting probability when $\fIII$ has type $\ell = 1,2,3$, which each occur with probability $\xi_{n-j-k,\ell}^{(i)}$; the lifting probability is 1 otherwise by Proposition \ref{prop:lifting_prob_1}.

    Repartitioning by setting $S' = S \cup W$, $T'=T$, and $U'=U$, we obtain a rotated version of Table \ref{tab:val_table_theta_init}, so the lifting probability is given by $\theta_{n, jk\ell}$. This final step is illustrated below in Table \ref{tab:val_table_rotate}. Putting everything together yields \eqref{eq:thetas}.
\end{proof}

\begin{table}[h]
    \centering
    \caption{Rotating and collapsing the valuation tetrahedron for $\fIII$ (color added for emphasis)}
    \label{tab:val_table_rotate}
    \begin{tabular}{c c c c | c c c | c c | c }
        \multicolumn{4}{c|}{$w=0$} & 
        \multicolumn{3}{c|}{$w=1$} &
        \multicolumn{2}{c|}{$w=2$} & $w=3$ \\ \hline 
         {\color{blue}$=0_i$} & {\color{red}$\geq 1$} & {\color{gray}$\geq 1$} & $=1_j$ & {\color{blue}$\geq 0$} & {\color{red}$\geq 1$} & {\color{gray}$\geq 1$} & {\color{blue}$\geq 0$} & {\color{red}$\geq 1$} & {\color{blue}$\geq 0$}\\
         {\color{red}$\geq 1$} & {\color{gray}$\geq 1$} & $\geq 2$ &     & {\color{red}$\geq 1$} & {\color{gray}$\geq 1$} &     & {\color{red}$\geq 1$} &     & \\
         {\color{gray}$\geq 2$} & $\geq 2$ &     &     & {\color{gray}$\geq 2$} &     &     &     &     & \\
         $=2_k$ &     &     &     &     &     &     &     &     & \\
    \end{tabular}\\[2ex]
    \begin{tabular}{r r c c c c l}
       &$S \cup W$ &&&&& $T$ \\
        && {\color{blue}$=0^{(i)}$} & {\color{red}$\geq 1$} & {\color{gray}$\geq 1$} & $=1_j$ \\
        $\xrightarrow{\text{reindex}}$ && {\color{red}$\geq 1$} & {\color{gray}$\geq 1$} & $\geq 2$ \\
        && {\color{gray}$\geq 2$} & $\geq 2$  \\
        && $=2_k$ \\
        &$U$        
    \end{tabular}
    \end{table}

\section{Obtaining explicit rational functions}
\label{sec:explicit_functions}

We are now ready to prove Theorem \ref{thm:rho_p_all_n}, that for $n \geq 1$ and all primes $p$, we have
\[\rho_n(p) = 1 - \frac{g_n(p)}{h_n(p)} \]
for the explicit polynomials below in \eqref{eq:g1} -- \eqref{eq:h8} when $1 \leq n \leq 8$, and $\rho_n(p) = 1$ for $n \geq 9$. Note that $\frac{g_n}{h_n}$ as presented here are not all in lowest terms, in order to make the expressions more compact.

{
\begin{dgroup*}
\setlength{\parskip}{1ex}
\setlength{\parindent}{0ex}
	\begin{dmath}\label{eq:g1}
		g_1(p) = (p^2 + 1)^2
	\end{dmath}
	\begin{dmath}
		h_1(p) = 3(p^4 + p^3 + p^2 + p + 1)
	\end{dmath}
    \begin{dmath}\label{eq:g2}
        g_2(p) = p^9 - p^8 + p^6 - p^4 + p^3 + p^2 - 2p + 1
    \end{dmath}
    \begin{dmath}
        h_2(p) = 3(p^6 + p^3 + 1)(p^4 + 1)(p^2 + 1) 
    \end{dmath}
    \begin{dmath}    
        g_3(p) = \left(3  p^{26} + p^{24} + p^{23} + 4  p^{22} - 3  p^{21} + 3  p^{20} + 2  p^{19} + 2  p^{18} - p^{17} + p^{14} + p^{13} - 2  p^{12} + 3  p^{11} + 3  p^{7}\right) \left(p^{2} + 1\right) \left(p + 1\right)^{2} \left(p - 1\right)^{4}
    \end{dmath}
    \begin{dmath}
        h_3(p) = 9  \left(p^{13} - 1\right) \left(p^{7} + 1\right) \left(p^{7} - 1\right) \left(p^{6} + 1\right) \left(p^{5} - 1\right) \left(p^{3} + 1\right) \left(p^{3} - 1\right)
    \end{dmath}
    \begin{dmath}\label{eq:g4}
        g_4(p) = \left(p^{46} + 3  p^{41} + p^{40} - p^{39} + p^{37} + p^{36} + p^{35} - 3  p^{34} + 3  p^{27} - p^{26} + p^{25} + p^{19}\right) \left(p^{2} + 1\right) \left(p + 1\right)^{2} \left(p - 1\right)^{4}
    \end{dmath}
    \begin{dmath}
        h_4(p) = 9  \left(p^{19} - 1\right) \left(p^{17} - 1\right) \left(p^{10} + 1\right) \left(p^{9} + 1\right) \left(p^{9} - 1\right) \left(p^{7} - 1\right) \left(p^{5} + 1\right)
    \end{dmath}
    \begin{dmath}
        g_5(p) = \left(3  p^{91} - 3  p^{90} + 3  p^{88} + 3  p^{85} - 3  p^{84} + 3  p^{82} - 3  p^{81} + 3  p^{79} + 3  p^{78} + 3  p^{76} - 3  p^{75} + 3  p^{73} - 2  p^{72} + p^{71} + 4  p^{70} - 3  p^{69} + 3  p^{67} - 3  p^{66} + 3  p^{64} - 3  p^{62} + 3  p^{61} + 3  p^{59} + 3  p^{58} - 3  p^{56} + 3  p^{55} - 3  p^{53} + 3  p^{52} + 3  p^{49} - 3  p^{47} + 3  p^{46} - 3  p^{44} + 3  p^{43} - 3  p^{41} + 3  p^{40} - 3  p^{38} + 3  p^{37}\right) \left(p^{5} - 1\right) \left(p^{2} + 1\right) \left(p + 1\right)^{2} \left(p - 1\right)^{4}
    \end{dmath}
    \begin{dmath}
        h_5(p) = 27  \left(p^{27} - 1\right) \left(p^{25} - 1\right) \left(p^{23} - 1\right) \left(p^{14} + 1\right) \left(p^{13} + 1\right) \left(p^{13} - 1\right) \left(p^{12} + 1\right) \left(p^{7} + 1\right) \left(p^{7} - 1\right) \left(p^{6} + 1\right)
    \end{dmath}
    \begin{dmath}
        g_6(p) = \left(3  p^{105} + p^{97} + p^{96} + p^{95} - 3  p^{93} + 3  p^{81}\right) \left(p + 1\right)^{2} \left(p - 1\right)^{7}
    \end{dmath}
    \begin{dmath}
        h_6(p) = 27  \left(p^{31} - 1\right) \left(p^{24} - p^{23} + p^{19} - p^{18} + p^{17} - p^{16} + p^{14} - p^{13} + p^{12} - p^{11} + p^{10} - p^{8} + p^{7} - p^{6} + p^{5} - p + 1\right) \left(p^{20} - p^{19} + p^{17} - p^{16} + p^{14} - p^{13} + p^{11} - p^{10} + p^{9} - p^{7} + p^{6} - p^{4} + p^{3} - p + 1\right) \left(p^{17} + 1\right) \left(p^{17} - 1\right) \left(p^{16} + 1\right) \left(p^{11} - 1\right) \left(p^{8} + p^{7} - p^{5} - p^{4} - p^{3} + p + 1\right) \left(p^{8} - p^{7} + p^{5} - p^{4} + p^{3} - p + 1\right) \left(p^{8} + 1\right) \left(p^{6} + 1\right) \left(p^{5} + 1\right) \left(p^{5} - 1\right) \left(p^{4} + 1\right) \left(p^{3} + 1\right) \left(p^{3} - 1\right)^{3}
    \end{dmath}
    \begin{dmath}
        g_7(p) = \left(p^{4} + 1\right) \left(p^{2} + 1\right)^{2} \left(p + 1\right)^{4} \left(p - 1\right)^{9} p^{141}
    \end{dmath}
    \begin{dmath}
        h_7(p) = 27  \left(p^{43} - 1\right) \left(p^{41} - 1\right) \left(p^{24} - p^{23} + p^{21} - p^{20} + p^{18} - p^{17} + p^{15} - p^{14} + p^{12} - p^{10} + p^{9} - p^{7} + p^{6} - p^{4} + p^{3} - p + 1\right) \left(p^{22} + 1\right) \left(p^{20} + 1\right) \left(p^{19} + 1\right) \left(p^{19} - 1\right) \left(p^{13} - 1\right) \left(p^{12} + p^{11} - p^{9} - p^{8} + p^{6} - p^{4} - p^{3} + p + 1\right) \left(p^{12} - p^{11} + p^{9} - p^{8} + p^{6} - p^{4} + p^{3} - p + 1\right) \left(p^{11} + 1\right) \left(p^{11} - 1\right) \left(p^{10} + 1\right) \left(p^{8} - p^{7} + p^{5} - p^{4} + p^{3} - p + 1\right) \left(p^{7} + 1\right) \left(p^{5} + 1\right) \left(p^{5} - 1\right) \left(p^{3} - 1\right)^{3}
    \end{dmath}
    \begin{dmath}
        g_8(p) = \left(p^{9} - 1\right) \left(p^{7} - 1\right) \left(p^{4} + 1\right) \left(p^{2} + 1\right)^{2} \left(p + 1\right)^{3} \left(p - 1\right)^{9} p^{219}
    \end{dmath}
    \begin{dmath}\label{eq:h8}
        h_8(p) = 27  \left(p^{53} - 1\right) \left(p^{49} - 1\right) \left(p^{47} - 1\right)\left(p^{40} - p^{39} + p^{35} - p^{34} + p^{30} - p^{28} + p^{25} - p^{23} + p^{20} - p^{17} + p^{15} - p^{12} + p^{10} - p^{6} + p^{5} - p + 1\right) \left(p^{32} - p^{31} + p^{29} - p^{28} + p^{26} - p^{25} + p^{23} - p^{22} + p^{20} - p^{19} + p^{17} - p^{16} + p^{15} - p^{13} + p^{12} - p^{10} + p^{9} - p^{7} + p^{6} - p^{4} + p^{3} - p + 1\right) \left(p^{27} + 1\right) \left(p^{27} - 1\right) \left(p^{26} + 1\right) \left(p^{25} + 1\right) \left(p^{25} - 1\right) \left(p^{24} + 1\right) \left(p^{17} - 1\right) \left(p^{13} + 1\right) \left(p^{13} - 1\right) \left(p^{12} + 1\right) \left(p^{11} - 1\right) \left(p^{6} + 1\right) \left(p^{3} - 1\right)^{3} 
    \end{dmath}
\end{dgroup*}
}

\begin{proof}[Proof of Theorem \ref{thm:rho_p_all_n}]

To calculate $\rho_{n}(p)$, we assemble the relations \eqref{eq:rho_structure_eq}, \eqref{eq:rho_(j)}, \eqref{eq:sigmas}, \eqref{eq:sigma_cond}, \eqref{eq:sigma_prime}, \eqref{eq:tau}, \eqref{eq:tau_prime}, and \eqref{eq:thetas} into a system of 64 linear equations in the 64 unknown lifting probabilities $\rho$, $\rho^{(i)}$, $\sigma_i$, $\sigma_i^{(k)}$, $\sigma_i'$, $\tau_{ij}$, $\tau_{ij}'$, $\theta_{ijk}$, where $i,j,k$ take values in $\{1,2,3\}$. The unique solutions in the field of rational functions in the transcendental $p$ yield the values for $\rho_n(p)$ described explicitly above in \eqref{eq:g1} -- \eqref{eq:h8} in terms of $g_n(p), h_n(p)$ for $1 \leq n \leq 8$. For $n = 9$ the unique solution is $\rho_9=1$ (and all other lifting probabilities are equal to 1). For $n > 9$, the conclusion is deduced by specializing, i.e.\ to a nonzero cubic form in only 10 of the $n+1$ variables, setting the rest to zero.

In practice, this was accomplished via a symbolic solve in \texttt{Sage}. While all 64 probabilities could be solved for at once, it is was significantly faster to observe sub-linear systems in fewer unknowns and exploit this by solving them first, thereby reducing the size of subsequent solves:
\begin{enumerate}
    \item Observe that the 27 equations \eqref{eq:thetas} in $\theta_{n,ijk}$ involve no other lifting probabilities.
    \item Once $\theta_{n,ijk}$ have been determined, the 27 equations $\eqref{eq:sigma_cond}, \eqref{eq:tau}, \eqref{eq:tau_prime}$ depend only on $\sigma_{n,i}^{(k)}$, $\tau_{n,ij}$, $\tau_{n,ij}'$.
    \item Solve for the remaining 10 probabilities $\rho_n$, $\rho_n^{(i)}$, $\sigma_{n,i}$, $\sigma_{n,i}'$ in the system \eqref{eq:rho_structure_eq}, \eqref{eq:rho_(j)}, \eqref{eq:sigmas}, \eqref{eq:sigma_prime}.
\end{enumerate}
An implementation of this approach is available in the file \texttt{compute_rho_p.ipynb} available in our GitHub repository \cite{BeneishKeyes_code_repo}. 
\end{proof}

\section{Numerical approximations}
\label{sec:numerics}

We now seek to obtain precise numerical approximations for $\rhoELS_n$, the density of everywhere locally soluble cubic hypersurfaces when $2 \leq n \leq 8$ (and hence $\rho_n$ for $n\geq 4$ by \cite{BLeBS})\footnote{Note that for $n=1$, we have $\rho_n(p) \sim \frac{2}{3}$ as $p \to \infty$, so the product $\prod_{p \leq A} \rho_{1}(p)$ goes to 0 in the limit as $A \to \infty$.}, by truncating, i.e.\
\[\rho_n \approx \prod_{p \leq A} \rho_n(p).\] To determine how large to make $A$ and precisely estimate the error, we relate $\rho_n$ to certain values of the Riemann zeta function and invoke classical techniques for approximating $\zeta(s)$.

From the explicit description of $\rho_n(p)$ given in \eqref{eq:g2} -- \eqref{eq:h8}, we extract the asymptotics 
\[1 - \rho_n(p) = \frac{g_n(p)}{h_n(p)} \sim \frac{1}{\gamma_n p^{\delta_n}}\]
for $2 \leq n \leq 8$ with $\gamma_n, \delta_n$ given below.
\begin{table}[h]
    \centering
    \caption{Asymptotics for $\frac{g_n(p)}{h_n(p)}$}
    \label{tab:asymptotics}
    \begin{tabular}{r c c }
         \hline $n$ & $\gamma_n$ & $\delta_n$ \\ \hline
         2 & 3 & 3 \\
         3 & 3 & 10 \\
         4 & 9 & 22 \\
         5 & 9 & 43 \\
         6 & 9 & 78 \\
         7 & 27 & 129 \\
         8 & 27 & 201\\ \hline
    \end{tabular}
\end{table}

An explicit check reveals that as a function on primes $p$, $\frac{g_n}{h_n}$ approaches this asymptotic from below.

\begin{lemma}\label{lem:comparing_g/h}
    For all $2 \leq n \leq 8$ and all primes $p$, we have
    \[\frac{g_n(p)}{h_n(p)} \leq \frac{1}{\gamma_np^{\delta_n}}.\]
\end{lemma}

Recall the Riemann zeta function and its usual Euler product, $\zeta(s) = \sum_{m \geq 1} \frac{1}{m^s} = \prod_p \left(1 - \frac{1}{p^s}\right)^{-1}$.
For $A$ a positive real number we denote by $\zeta_{>A}(s)$ the tail of the Euler product,
\[\zeta_{>A}(s) = \prod_{p > A}\left(1 - \frac{1}{p^s}\right)^{-1}.\]
It turns out that $\prod_p \rho_n(p)$ is close to its truncation when $\zeta_{>A}(\delta_n)$ is sufficiently close to 1.

\begin{lemma}\label{lem:rho_truncation_error}
    Fix $2 \leq n \leq 8$. Let $B \geq 1$ be a real number. If $\zeta_{>A}(\delta_n) \leq B$, then we have
    \[\left|\rho_n^{\mathrm{ELS}} - \prod_{p \leq A}\rho_n(p)\right| \leq 1 - \frac{1}{B^{1/\gamma_n}}\]
\end{lemma}

\begin{proof}
    We drop the subscripts $n$ for brevity. We have
    \begin{align*}
        \left|\rho^{\mathrm{ELS}} - \prod_{p \leq A}\rho(p)\right| &= \left(\prod_{p \leq A}\rho(p)\right) \left(1 - \prod_{p > A}\rho(p)\right) \\
        &\leq 1 - \prod_{p > A}\rho(p)
    \end{align*}
    since $\rho(p)$ is at most 1. 
    
    Suppose we would like to show that
    \begin{equation}\label{eq:epsilon1}
        1 - \prod_{p > A}\rho(p) \leq \epsilon_1.
    \end{equation}
    This is equivalent to 
    \begin{equation}\label{eq:epsilon2}
        \log\left(\prod_{p > A}\rho(p)\right) \geq \epsilon_2 = \log(1 - \epsilon_1).
    \end{equation}
    Note that the quantities in \eqref{eq:epsilon2} are negative. By the Taylor series for the logarithm and absolute convergence of the involved sums, we may write
    \begin{align*}
        \log\left(\prod_{p > A}\rho(p)\right) &= \sum_{p > A} \sum_{j \geq 1} \frac{-(g/h)^j}{j}\\
        & \geq \sum_{p > A} \sum_{j \geq 1} \frac{-1}{j\gamma^jp^{\delta j}} &\text{(by Lemma \ref{lem:comparing_g/h})} \\
        % & \geq \frac1\gamma \sum_{p > A} \sum_{j \geq 1} \frac{-1}{jp^{\delta j}}\\
        % & = \frac1\gamma \log\left(\prod_{p > A} \left(1 - \frac{1}{p^\delta}\right)\right)\\
        & \geq \frac{-1}{\gamma} \log\left( \zeta_{>A}(\delta) \right).
    \end{align*}
    Thus to establish \eqref{eq:epsilon2}, it suffices to show that 
    \begin{equation}\label{eq:epsilon3}
        \frac{1}{\gamma} \log\left(\zeta_{>A}(\delta)\right) \leq \epsilon_3 = -\epsilon_2.
    \end{equation}
    Finally, we rearrange \eqref{eq:epsilon3} to give the equivalent
    \begin{equation}\label{eq:epsilon4}
        \zeta_{>A}(\delta) \leq \epsilon_4 = e^{\gamma\epsilon_3}.
    \end{equation}

    Suppose we know $\zeta_{>A}(\delta) < B$ for some choice of $A$. Note that such a choice exists for $B$ arbitrarily close to 1, since $\zeta_{>A}(\delta)$ is the tail of a convergent infinite product. Tracing through \eqref{eq:epsilon1} -- \eqref{eq:epsilon4}, we establish \eqref{eq:epsilon1} with $\epsilon_1 = 1 - B^{-1/\gamma}$, from which the lemma follows.    
\end{proof}

An upper estimate for $\zeta_{>A}(s)$ follows from the classical Euler--Maclaurin summation formula applied to $\zeta(s)$; see e.g. \cite{Apostol_Euler_summation, cohen}.

\begin{lemma}\label{lem:zeta_euler_mac}
    Fix a real number $A \geq 1$, an integer $M \geq 1$, an even integer $I \geq 2$, and let $B_{2i}$ denote the $2i$-th Bernoulli number. For integral $s \geq 2$ we have
    \begin{equation*}\label{eq:zeta_>A_bound}
        \zeta_{>A}(s) \leq \scalebox{0.85}{$\displaystyle \left(\prod_{p \leq A}\left(1 - \frac{1}{p^s}\right) \right) 
            \left( \sum_{m=1}^M \frac{1}{m^s} + \frac{1}{(s-1)M^{s-1}} - \frac{1}{2M^s} + \sum_{i=1}^I \frac{B_{2i}(s+2i - 2)!}{(2i)!(s-1)!M^{s+2i-1}} + \left| \frac{B_{2I + 2}(s+2I)!}{(2I + 2)!(s-1)!M^{s+2I + 1}} \right| \right)$}.
    \end{equation*}
\end{lemma}

\begin{remark}
    The upper bound for $\zeta_{>A}(s)$ has the advantage of being able to be computed via exact arithmetic in \texttt{Sage}. Then when combining with Lemma \ref{lem:rho_truncation_error} we convert to inexact arithmetic only at the final step, choosing to round up. See \texttt{rho_numerics.ipynb} in our GitHub repository \cite{BeneishKeyes_code_repo}.
\end{remark}

Suppose we want our truncation $\prod_{p \leq A} \rho_n(p)$ to be accurate within $10^{-D}$ of the true value of $\rho_n^{\mathrm{ELS}}$. Then we set $B = \left( \frac{10^{D}}{10^{D} - 1} \right)^{\gamma_n}$ and apply Lemmas \ref{lem:rho_truncation_error} and \ref{lem:zeta_euler_mac} with various values of $D$ and $A$ with $M=1000, I=4$ to obtain the accuracy data in Table \ref{tab:precision}. 
    \begin{table}[h]
        \centering
        \caption{Accuracy of truncations}
        \label{tab:precision}
        \begin{tabular}{| c | c | c | c |}
             \hline $n$ & $A$ & $1 - \prod_{p \leq A} \rho_n(p) \approx$ & $D$ \\ \hline 
             2 & 61 & 0.0274 & 5  \\ 
              & 12919 &  & 10  \\ \hline
             3 & 11 & 0.00007328 & 10  \\
              & 503 &  & 26  \\ \hline
             4 & 5 & $5.022 \cdot 10^{-9}$ & 16  \\ 
              & 179 &  & 50 \\ \hline
             5 & 3 & $1.343 \cdot 10^{-15}$ & 21  \\ 
              & 17 &  & 53 \\ \hline
             % 5 & 227 &  & 100 \\ \hline
             6 & 3 & $3.502 \cdot 10^{-26}$ & 38 \\ 
              & 19 & & 100 \\ \hline
             7 & 3 & $5.152 \cdot 10^{-42}$ & 62 \\ 
              & 7 & & 110 \\ \hline
             8 & 3 & $6.222 \cdot 10^{-64}$ & 97  \\ 
              & 5 &  & 141 \\ \hline
        \end{tabular}
    \end{table}

\bibliography{cubic_hypersurfaces}
\bibliographystyle{alpha}

\end{document}